\documentclass[12pt]{amsart}
\usepackage{txfonts}      %{article} was 12pt latex e
\usepackage{amssymb}
\usepackage{eucal}
\usepackage{amsmath}
\usepackage{amscd}
\usepackage{xcolor}
\usepackage{multicol}
\usepackage[all]{xy}           %xypic macro for latex2.09
\usepackage{graphicx}
\usepackage{color}
\usepackage{colordvi}
\usepackage{xspace}
\usepackage{tikz}
\usepackage{makecell}
\usepackage{appendix}
\usepackage{amsthm}
\usepackage{enumerate}
\usepackage{longtable}

\usepackage{ifpdf}
\ifpdf
\usepackage[colorlinks,final,backref=page,hyperindex]{hyperref}
\else
\usepackage[colorlinks,final,backref=page,hyperindex,hypertex]{hyperref}
\fi

\usepackage[active]{srcltx} %SRC Specials for DVI Searching

\begin{document}

%%%%%%%%%%%%%%%%%%%%%%%% Statements

\newtheorem{thm}{Theorem}[section]
\newtheorem{lem}[thm]{Lemma}
\newtheorem{cor}[thm]{Corollary}
\newtheorem{pro}[thm]{Proposition}
\theoremstyle{definition}
\newtheorem{defi}[thm]{Definition}
\newtheorem{ex}[thm]{Example}
\newtheorem{rmk}[thm]{Remark}
\newtheorem{pdef}[thm]{Proposition-Definition}
\newtheorem{condition}[thm]{Condition}

\renewcommand{\labelenumi}{{\rm(\alph{enumi})}}
\renewcommand{\theenumi}{\alph{enumi}}

\newcommand {\emptycomment}[1]{} %to remove paragraphs

\newcommand{\nc}{\newcommand}
\newcommand{\delete}[1]{}
\newcommand{\xrightleftharpoons}[2]{\mathrel{\mathop{\rightleftharpoons}\limits_{#2}^{#1}}}

\nc{\tred}[1]{\textcolor{red}{#1}}
\nc{\tblue}[1]{\textcolor{blue}{#1}}
\nc{\tgreen}[1]{\textcolor{green}{#1}}
\nc{\tpurple}[1]{\textcolor{purple}{#1}}
\nc{\tgray}[1]{\textcolor{gray}{#1}}
\nc{\torg}[1]{\textcolor{orange}{#1}}
\nc{\tmag}[1]{\textcolor{magenta}}
\nc{\btred}[1]{\textcolor{red}{\bf #1}}
\nc{\btblue}[1]{\textcolor{blue}{\bf #1}}
\nc{\btgreen}[1]{\textcolor{green}{\bf #1}}
\nc{\btpurple}[1]{\textcolor{purple}{\bf #1}}

\nc{\todo}[1]{\tred{To do:} #1}

%\delete{
    \nc{\mlabel}[1]{\label{#1}}  % Use this to suppress names
    \nc{\mcite}[1]{\cite{#1}}  % Use this to suppress names
    \nc{\mref}[1]{\ref{#1}}  % Use this to suppress names
    \nc{\meqref}[1]{\eqref{#1}}  % Use this to suppress names
    \nc{\mbibitem}[1]{\bibitem{#1}} % Use this to show number
%}

\delete{
%    \nc{\mlabel}[1]{\label{#1}  % Use the next two lines to show names
%        { {\small\tgreen{\tt{{\ }(#1)}}}}}
    \nc{\mcite}[1]{\cite{#1}{\small{\tt{{\ }(#1)}}}}  % Use this lines to show names
    \nc{\mref}[1]{\ref{#1}{\small{\tred{\tt{{\ }(#1)}}}}}  % Use this lines to show names
    \nc{\meqref}[1]{\eqref{#1}{{\tt{{\ }(#1)}}}}  % Use this lines to show names
    \nc{\mbibitem}[1]{\bibitem[\bf #1]{#1}} % Use this to show name
}

%    \nc{\mlabel}[1]{  % Use the next two lines to show names
%           { {\small\tgreen{\tt{{\ }(#1)}}}}}

\nc{\cm}[1]{\textcolor{red}{Chengming:#1}}
\nc{\yy}[1]{\textcolor{blue}{Yanyong: #1}}
%\nc{\lit}[2]{\textcolor{blue}{#1}{ \textcolor{purple}{(#2)}}}
%\nc{\lit}[2]{\textcolor{blue}{#1}{}} %use this line instead of the previous one to show only the new changes
\nc{\li}[1]{\textcolor{purple}{#1}}
\nc{\lir}[1]{\textcolor{purple}{Li:#1}}

\nc{\revise}[1]{\textcolor{blue}{#1}}

%%%%%%%% new symbols

\nc{\tforall}{\ \ \text{for all }}
\nc{\hatot}{\,\widehat{\otimes} \,}
\nc{\complete}{completed\xspace}
\nc{\wdhat}[1]{\widehat{#1}}

\nc{\ts}{\mathfrak{p}}
\nc{\mts}{c_{(i)}\ot d_{(j)}}

\nc{\NA}{{\bf NA}}
\nc{\LA}{{\bf Lie}}
\nc{\CLA}{{\bf CLA}}

\nc{\cybe}{CYBE\xspace}
\nc{\nybe}{NYBE\xspace}
\nc{\ccybe}{CCYBE\xspace}

\nc{\ndend}{pre-Novikov\xspace}
\nc{\calb}{\mathcal{B}}
\nc{\rk}{\mathrm{r}}
\newcommand{\g}{\mathfrak g}
\newcommand{\h}{\mathfrak h}
\newcommand{\pf}{\noindent{$Proof$.}\ }
\newcommand{\frkg}{\mathfrak g}
\newcommand{\frkh}{\mathfrak h}
\newcommand{\Id}{\rm{Id}}
\newcommand{\gl}{\mathfrak {gl}}
\newcommand{\ad}{\mathrm{ad}}
\newcommand{\add}{\frka\frkd}
\newcommand{\frka}{\mathfrak a}
\newcommand{\frkb}{\mathfrak b}
\newcommand{\frkc}{\mathfrak c}
\newcommand{\frkd}{\mathfrak d}
\newcommand {\comment}[1]{{\marginpar{*}\scriptsize\textbf{Comments:} #1}}

\nc{\vspa}{\vspace{-.1cm}}
\nc{\vspb}{\vspace{-.2cm}}
\nc{\vspc}{\vspace{-.3cm}}
\nc{\vspd}{\vspace{-.4cm}}
\nc{\vspe}{\vspace{-.5cm}}

%%%%%%%%%%%%%%%%%%%%%%% old symbols

\nc{\disp}[1]{\displaystyle{#1}}
\nc{\bin}[2]{ (_{\stackrel{\scs{#1}}{\scs{#2}}})}  %binomial coeff
\nc{\binc}[2]{ \left (\!\! \begin{array}{c} \scs{#1}\\
    \scs{#2} \end{array}\!\! \right )}  %binomial coeff
\nc{\bincc}[2]{  \left ( {\scs{#1} \atop
    \vspace{-.5cm}\scs{#2}} \right )}  %binomial coeff
\nc{\ot}{\otimes}
\nc{\sot}{{\scriptstyle{\ot}}}
\nc{\otm}{\overline{\ot}}
\nc{\ola}[1]{\stackrel{#1}{\la}}%${\Bbb Z}$

\nc{\scs}[1]{\scriptstyle{#1}} \nc{\mrm}[1]{{\rm #1}}

\nc{\dirlim}{\displaystyle{\lim_{\longrightarrow}}\,}
\nc{\invlim}{\displaystyle{\lim_{\longleftarrow}}\,}

\nc{\bfk}{{\bf k}} \nc{\bfone}{{\bf 1}}
\nc{\rpr}{\circ}
%\nc{\apr}{\cdot}
\nc{\dpr}{{\tiny\diamond}}
\nc{\rprpm}{{\rpr}}

%%%%%%%%%%%%%%%%%%%%% roman fonts, in alphabetic order
\nc{\mmbox}[1]{\mbox{\ #1\ }} \nc{\ann}{\mrm{ann}}
\nc{\Aut}{\mrm{Aut}} \nc{\can}{\mrm{can}}
\nc{\twoalg}{{two-sided algebra}\xspace}
\nc{\colim}{\mrm{colim}}
\nc{\Cont}{\mrm{Cont}} \nc{\rchar}{\mrm{char}}
\nc{\cok}{\mrm{coker}} \nc{\dtf}{{R-{\rm tf}}} \nc{\dtor}{{R-{\rm
tor}}}
\renewcommand{\det}{\mrm{det}}
\nc{\depth}{{\mrm d}}
\nc{\End}{\mrm{End}} \nc{\Ext}{\mrm{Ext}}
\nc{\Fil}{\mrm{Fil}} \nc{\Frob}{\mrm{Frob}} \nc{\Gal}{\mrm{Gal}}
\nc{\GL}{\mrm{GL}} \nc{\Hom}{\mrm{Hom}} \nc{\hsr}{\mrm{H}}
\nc{\hpol}{\mrm{HP}}  \nc{\id}{\mrm{id}} \nc{\im}{\mrm{im}}

\nc{\incl}{\mrm{incl}} \nc{\length}{\mrm{length}}
\nc{\LR}{\mrm{LR}} \nc{\mchar}{\rm char} \nc{\NC}{\mrm{NC}}
\nc{\mpart}{\mrm{part}} \nc{\pl}{\mrm{PL}}
\nc{\ql}{{\QQ_\ell}} \nc{\qp}{{\QQ_p}}
\nc{\rank}{\mrm{rank}} \nc{\rba}{\rm{RBA }} \nc{\rbas}{\rm{RBAs }}
\nc{\rbpl}{\mrm{RBPL}}
\nc{\rbw}{\rm{RBW }} \nc{\rbws}{\rm{RBWs }} \nc{\rcot}{\mrm{cot}}
\nc{\rest}{\rm{controlled}\xspace}
\nc{\rdef}{\mrm{def}} \nc{\rdiv}{{\rm div}} \nc{\rtf}{{\rm tf}}
\nc{\rtor}{{\rm tor}} \nc{\res}{\mrm{res}} \nc{\SL}{\mrm{SL}}
\nc{\Spec}{\mrm{Spec}} \nc{\tor}{\mrm{tor}} \nc{\Tr}{\mrm{Tr}}
\nc{\mtr}{\mrm{sk}}

%%%%%%%%%%%%%%%%%% bold face
\nc{\ab}{\mathbf{Ab}} \nc{\Alg}{\mathbf{Alg}}

%%%%%%%%%%%%%%%%%%%Bbb fonts
\nc{\BA}{{\mathbb A}} \nc{\CC}{{\mathbb C}} \nc{\DD}{{\mathbb D}}
\nc{\EE}{{\mathbb E}} \nc{\FF}{{\mathbb F}} \nc{\GG}{{\mathbb G}}
\nc{\HH}{{\mathbb H}} \nc{\LL}{{\mathbb L}} \nc{\NN}{{\mathbb N}}
\nc{\QQ}{{\mathbb Q}} \nc{\RR}{{\mathbb R}} \nc{\BS}{{\mathbb{S}}} \nc{\TT}{{\mathbb T}}
\nc{\VV}{{\mathbb V}} \nc{\ZZ}{{\mathbb Z}}

%%%%%%%%%%%%%%%%%%% cal fonts

\nc{\calao}{{\mathcal A}} \nc{\cala}{{\mathcal A}}
\nc{\calc}{{\mathcal C}} \nc{\cald}{{\mathcal D}}
\nc{\cale}{{\mathcal E}} \nc{\calf}{{\mathcal F}}
\nc{\calfr}{{{\mathcal F}^{\,r}}} \nc{\calfo}{{\mathcal F}^0}
\nc{\calfro}{{\mathcal F}^{\,r,0}} \nc{\oF}{\overline{F}}
\nc{\calg}{{\mathcal G}} \nc{\calh}{{\mathcal H}}
\nc{\cali}{{\mathcal I}} \nc{\calj}{{\mathcal J}}
\nc{\call}{{\mathcal L}} \nc{\calm}{{\mathcal M}}
\nc{\caln}{{\mathcal N}} \nc{\calo}{{\mathcal O}}
\nc{\calp}{{\mathcal P}} \nc{\calq}{{\mathcal Q}} \nc{\calr}{{\mathcal R}}
\nc{\calt}{{\mathcal T}} \nc{\caltr}{{\mathcal T}^{\,r}}
\nc{\calu}{{\mathcal U}} \nc{\calv}{{\mathcal V}}
\nc{\calw}{{\mathcal W}} \nc{\calx}{{\mathcal X}}
\nc{\CA}{\mathcal{A}}

%%%%%%%%%%%%%%%%%%  frak fonts
\nc{\fraka}{{\mathfrak a}} \nc{\frakB}{{\mathfrak B}}
\nc{\frakb}{{\mathfrak b}} \nc{\frakd}{{\mathfrak d}}
\nc{\oD}{\overline{D}}
\nc{\frakF}{{\mathfrak F}} \nc{\frakg}{{\mathfrak g}}
\nc{\frakm}{{\mathfrak m}} \nc{\frakM}{{\mathfrak M}}
\nc{\frakMo}{{\mathfrak M}^0} \nc{\frakp}{{\mathfrak p}}
\nc{\frakS}{{\mathfrak S}} \nc{\frakSo}{{\mathfrak S}^0}
\nc{\fraks}{{\mathfrak s}} \nc{\os}{\overline{\fraks}}
\nc{\frakT}{{\mathfrak T}}
\nc{\oT}{\overline{T}}
%\nc{\frakx}{{\mathfrak x}}
\nc{\frakX}{{\mathfrak X}} \nc{\frakXo}{{\mathfrak X}^0}
\nc{\frakx}{{\mathbf x}}
%\nc{\frakTxo}{{\frakTx}^0}
\nc{\frakTx}{\frakT}      %All rooted trees, correspond to \ncsha(X)
\nc{\frakTa}{\frakT^a}        % rooted trees for \ncsha(A)
\nc{\frakTxo}{\frakTx^0}   % rooted trees for \ncshao(X)
\nc{\caltao}{\calt^{a,0}}   % rooted trees for \ncshao(A)
\nc{\ox}{\overline{\frakx}} \nc{\fraky}{{\mathfrak y}}
\nc{\frakz}{{\mathfrak z}} \nc{\oX}{\overline{X}}

\font\cyr=wncyr10

%%%%%%%%%%%%%%%%%%%%%%%%%%%%%%%%%%%%%%%%%%%%%%%%%%%%%%%%%%%%%%%%%%

%\begin{document}
\title[]{On $q$-pre-Lie algebras}

\author{Chengyang Lu}
\address{School of Mathematics, Hangzhou Normal University,
Hangzhou 311121, PR China}
\email{2024111029011@stu.hznu.edu.cn}

\author{Yanyong Hong~~(corresponding author)}
\address{School of Mathematics, Hangzhou Normal University,
Hangzhou 311121, PR China}
\email{yyhong@hznu.edu.cn}

\subjclass[2010]{
17A30, % Nonassociative algebras satisfying other identities
17B65, %  Infinite-dimensional Lie (super)algebras
17A60, %  Structure theory for nonassociative algebras
17D25  % Lie-admissible algebras
}
\keywords{$q$-pre-Lie algebra, $q$-Novikov algebra, $q$-$\mathcal{O}$-operator, \delete{$q$-symmetric $2$-cocycle,} Witt algebra, Virasoro algebra.}

\begin{abstract}
   In this paper, we introduce the notion of $q$-pre-Lie algebras from the perspective of representations of Lie algebras, providing a parametrized generalization that unifies pre-Lie algebras and anti-pre-Lie algebras. For a $q$-pre-Lie algebra $(A,\circ)$, the commutator of $\circ$ is a Lie bracket and the left multiplication operator scaled by $q$ gives a representation of the associated commutator Lie algebra. We also introduce the notions of $q$-$\mathcal{O}$-operators and $q$-Novikov algebras, and investigate their relationships with $q$-pre-Lie algebras. Several explicit constructions of $q$-pre-Lie algebras are provided. Moreover, we give a complete classification of graded $q$-pre-Lie algebra structures on the Witt algebra and investigate a class of compatible $q$-pre-Lie algebra structures on the Virasoro algebra. Finally, we classify all compatible $q$-pre-Lie algebra structures on finite-dimensional complex simple Lie algebras. We show that every finite-dimensional complex simple Lie algebra admits a compatible $q$-pre-Lie algebra structure precisely when $q=2$, except for $\mathfrak{sl}_2(\mathbb{C})$, which also admits such a structure when $q=-1$. In particular, the compatible anti-pre-Lie algebra structures  on $\mathfrak{sl}_2(\mathbb{C})$ are parametrized by invertible anti-derivations.
\end{abstract}

\maketitle

\vspace{-1.2cm}

\tableofcontents

\vspace{-1.2cm}

\allowdisplaybreaks

\section{Introduction}
The notion of pre-Lie algebras, also known as left-symmetric algebras or Vinberg algebras, arose from the study of affine structures on Lie groups \cite{Ko}, convex homogeneous cones \cite{V} and the deformation theory of associative algebras \cite{Ge}, and has since appeared in many fields of mathematics and physics, such as vertex algebras \cite{BK,BLP}. Recall that a {\bf pre-Lie algebra} is a pair $(A,\circ)$, where $A$ is a vector space equipped with a binary operation satisfying
\begin{eqnarray*}
    (a\circ b)\circ c-a\circ(b\circ c)=(b\circ a)\circ c-b\circ(a\circ c),\qquad a,b,c\in A.
\end{eqnarray*}
A key property of pre-Lie algebras is that the commutator $[a,b]:=a\circ b-b\circ a$ defines a Lie algebra structure on $A$, and the left multiplication operator $L_\circ$ gives a representation of the associated commutator Lie algebra.

In recent years, the study of variants and generalizations of pre-Lie algebras has attracted increasing attention. A particularly notable variant is that of {\bf anti-pre-Lie algebras} \cite{LB}, which are binary operations $\circ$ such that the commutator $[\cdot,\cdot]$ is a Lie bracket and the negative left multiplication operator $-L_\circ$ is a representation of the associated Lie algebra. This representation-theoretic characterization of anti-pre-Lie algebras stands in direct parallel to the classical property of pre-Lie algebras, where $L_\circ$ itself is a representation. Motivated by this parallel, it is natural to seek a unified framework that interpolates between these two structures.

Following this idea, we introduce the notion of {\bf $q$-pre-Lie algebras} (see Definition 2.1): a binary operation $\circ$ such that the commutator $[\cdot,\cdot]$ is a Lie bracket and the scaled left multiplication operator $qL_\circ$ is a representation of the associated Lie algebra. This parametrized framework unifies pre-Lie algebras (the case $q=1$) and anti-pre-Lie algebras (the case $q=-1$), and provides a systematic approach to studying Lie-admissible algebras with scaled representations. We note that our definition differs from $\delta$-pre-Lie algebras defined in \cite{Ka}. When $q\neq 1$, a $q$-pre-Lie algebra is just a $\frac{1}{q}$-pre-Lie algebra defined in \cite{Ka} and their $1$-pre-Lie algebra is not the usual pre-Lie algebra. From the representation-theoretic perspective, our definition arises naturally by requiring the scaled left multiplication to be a representation. This perspective also connects our work closely to the recent framework of generalized splitting of algebras developed in \cite{BGLZ}, where the unified treatment of $\mathcal{O}$-operators and their algebraic consequences plays a central role. Our definition can be viewed as a special case of their type-$\begin{pmatrix} q & 0 \\ 0 & q \end{pmatrix}$ pre-Lie algebras, and our subsequent development of $q$-$\mathcal{O}$-operators and their relationship with $q$-pre-Lie algebras fits naturally into their general splitting theory.

Our first main result establishes a fundamental characterization: $(A,\circ)$ is a $q$-pre-Lie algebra if and only if $(A,[\cdot,\cdot])$ is a Lie algebra and $(A,qL_\circ)$ is a representation of it.\delete{When $q=0$, only the forward implication holds, as the zero map is trivially a representation.} This bridges the new notion with classical representation theory. Note that for a Lie algebra $(\mathfrak{g}, [\cdot,\cdot])$, there is a natural compatible $2$-pre-Lie algebra $(\mathfrak{g}, \circ)$ on $(\mathfrak{g}, [\cdot,\cdot])$ given by $a\circ b:=\frac{1}{2}[a,b]$ for all $a$, $b\in \mathfrak{g}$. We further provide several explicit constructions of $q$-pre-Lie algebras, including those arising from symmetric bilinear forms and from derivations and $q$-derivations of commutative associative algebras. We then introduce the notion of $q$-$\mathcal{O}$-operators on Lie algebras, which generalizes the classical $\mathcal{O}$-operators. A key result is that a $q$-$\mathcal{O}$-operator induces a $q$-pre-Lie algebra structure on the representation space if and only if it is strong. Moreover, we establish that \delete{when $q\neq 0$, }a Lie algebra admits a compatible $q$-pre-Lie algebra structure if and only if it admits an invertible $q$-$\mathcal{O}$-operator.  In addition, we prove that the existence of a nondegenerate invariant bilinear form on a $q$-pre-Lie algebra is equivalent to the isomorphism of two special representations of its associated Lie algebra.

Furthermore, following the approach of \cite{LB}, we explore the relationship between $q$-pre-Lie algebras and $q$-Novikov algebras. For any $p$ such that $p^2-p+1\neq 0$, a certain subclass of $(1-p)$-pre-Lie algebras, namely the $(1-p)$-Novikov algebras, corresponds to Novikov algebras via the $p$-algebra construction \cite{Dz}. Novikov algebras, as a subclass of pre-Lie algebras, arose independently in the study of Hamiltonian operators in formal variational calculus \cite{GD, GD2} and Poisson brackets of hydrodynamic type \cite{BN}. For each $p$ satisfying $p^2-p+1\neq0$ and $p\neq -1,0, 1$, the relationships among $(1-p)$-pre-Lie algebras, $(1-p)$-Novikov algebras, Novikov algebras and pre-Lie algebras are summarized as follows:
\begin{eqnarray*}
        \{\text{pre-Lie}\}\hookleftarrow\{\text{Novikov}\}\xrightleftharpoons{p\text{-algebra}}{(-p)\text{-algebra}}\{(1-p)\text{-Novikov} \}\hookrightarrow  \{(1-p)\text{-pre-Lie} \}.
    \end{eqnarray*}
Via the correspondence linking Novikov algebras to infinite-dimensional Lie algebras \cite{BN}, we obtain a correspondence between $(1-p)$-Novikov algebras and infinite-dimensional Lie algebras for $p\neq 0,\pm1$ with $p^2-p+1\neq0$. We also provide constructions of $q$-Novikov algebras from admissible pairs on commutative associative algebras, and consequently obtain new examples of $q$-pre-Lie algebras under suitable conditions.

Next, we turn to consider compatible $q$-pre-Lie algebra structures on the Witt algebra and the Virasoro algebra. Note that compatible pre-Lie algebra structures and anti-pre-Lie algebra structures on the Witt algebra and the Virasoro algebra have been investigated in \cite{KCB} and \cite{BG} respectively. Following their methods  using the classification of indecomposable weight representations of the Witt algebra, we prove that when $q\neq 1$, every graded $q$-pre-Lie algebra structure on the Witt algebra is parametrized by a complex parameter $\lambda$ and is given explicitly by
\[
W_n\circ W_m=\frac{1}{q}(\lambda+m+(1-q)n)W_{m+n},\qquad m,n\in\mathbb{Z}.
\]
For the Virasoro algebra, we investigate a class of compatible $q$-pre-Lie algebra structures satisfying
a certain multiplication condition. In contrast to the pre-Lie case (i.e., $q=1$) studied in \cite{KCB} and the case $q=2$, \delete{where graded compatible structures do exist under natural conditions, }we prove that no such structures exist when $q\neq 1,2$. This distinction highlights the essential role played by the scaling parameter $q$ in the representation condition.

Finally, we determine all compatible $q$-pre-Lie algebra structures on finite-dimensional complex simple Lie algebras $\mathfrak g$. When $\mathfrak g$ is not isomorphic to $\mathfrak{sl}_2(\mathbb C)$, we show that $q$ must be $2$ and that the multiplication is uniquely determined by $a\circ b=\frac12[a,b]$ for any $a,b\in\mathfrak g$. In the case of $\mathfrak{sl}_2(\mathbb C)$, such structures exist if and only if $q=2$ or $q=-1$. For $q=2$, the structure is given by $a\circ b=\frac12[a,b]$ for any $a,b\in\mathfrak{sl}_2(\mathbb{C})$. For $q=-1$, we give an explicit classification of all such structures in terms of invertible $(-1)$-derivations and obtain an isomorphism criterion for these structures. Note that compatible root-graded $(-1)$-pre-Lie algebra structures on $\mathfrak{sl}_2(\mathbb{C})$ have been classified in \cite{BG2}. Here, we show that there exist compatible $(-1)$-pre-Lie algebra structures on $\mathfrak{sl}_2(\mathbb{C})$ which are not root-graded.

The paper is organized as follows. Section 2 introduces the basic definitions and properties of $q$-pre-Lie algebras, including their relations with $q$-$\mathcal{O}$-operators and $q$-Novikov algebras, and provides several explicit constructions. Section 3 is devoted to investigating compatible $q$-pre-Lie algebra structures on the Witt algebra and the Virasoro algebra. Section 4 treats compatible $q$-pre-Lie algebra structures on finite-dimensional complex simple Lie algebras.\\

\noindent {\bf Notations.} Throughout this paper, let  $\bf k$ be a field of characteristic zero. In Sections 3 and 4, we take $\mathbf{k}=\mathbb{C}$, where $\mathbb{C}$ is the field of complex numbers. All vector spaces and algebras are over ${\bf k}$. All tensors over ${\bf k}$ are denoted by $\otimes$.  Let ${\bf k}^\times$ and ${\mathbb{C}}^\times$ be the sets of all nonzero elements of ${\bf k}$ and $\mathbb{C}$ respectively. Denote by $\mathbb{Z}$ and $\mathbb{Z}^\times$ the sets of integer numbers and non-zero integers respectively. We denote the identity map by $\id$. Let $A$ be a vector space with a binary operation $*$ and $a\in A$. Define two linear maps $L_*(a),R_*(a)\in{\rm End}_{\bf k}(A)$ by
\begin{eqnarray*}
    L_*(a)(b)=a*b,\;\;R_*(a)(b)=b*a,\;\;b\in A.
\end{eqnarray*}

\section{Basic results of $q$-pre-Lie algebras and the relationships with $q$-Novikov algebras}
In this section, we develop the basic theory of $q$-pre-Lie algebras. We introduce the definition, establish their fundamental characterization in terms of representations of Lie algebras, and explore their connections with $q$-$\mathcal{O}$-operators and $q$-Novikov algebras. Several explicit constructions are also provided.

\subsection{$q$-pre-Lie algebras and some constructions}
\begin{defi}
Let $q\in {\bf k}^\times$ and $A$ be a vector space with a binary operation $\circ: A\otimes A\rightarrow A$. $(A, \circ)$ is called a {\bf $q$-pre-Lie algebra} if the following equalities hold:
\begin{eqnarray}
\label{q-pre-Lie-1}&&q(a\circ (b\circ c)-b\circ (a\circ c))=[a,b]\circ c,\\
\label{q-pre-Lie-2}&&(q-1)([a,b]\circ c+[b,c]\circ a+[c,a]\circ b)=0,\;\;a, b, c\in A,
\end{eqnarray}
where $[a,b]:=a\circ b-b\circ a$.
\end{defi}
\begin{rmk}
\begin{itemize}
\item[(1)] When $q=1$, Eq. (\ref{q-pre-Lie-2}) naturally holds. Therefore, a $1$-pre-Lie algebra is just a pre-Lie algebra. \delete{When $q=0$, we have $[a,b]\circ c=0$ for all $a$, $b$, $c\in A$ and then Eq. (\ref{q-pre-Lie-2}) naturally holds.}
\item[(2)] When $q\neq 1$, Eq. (\ref{q-pre-Lie-2}) is just
\begin{eqnarray}
[a,b]\circ c+[b,c]\circ a+[c,a]\circ b=0,\;\;a, b, c\in A.
\end{eqnarray}
In this case, for each $q\neq0$, the $\frac{1}{q}$-pre-Lie algebra reduces to that of \cite[Definition 13]{Ka}.
In particular, a $(-1)$-pre-Lie algebra is just an anti-pre-Lie algebra as defined in \cite{LB}.
\item[(3)] If $(A,\circ)$ is a $q$-pre-Lie algebra, one can show that $(A, [\cdot,\cdot])$ is a Lie algebra. Therefore $(A,\circ)$ is Lie-admissible and $(A, [\cdot,\cdot])$ is called the {\bf associated Lie algebra} of $(A,\circ)$ and $(A,\circ)$ is called the {\bf compatible $q$-pre-Lie algebra structure} on $(A, [\cdot,\cdot])$.
\item[(4)] When $q\neq 0,1$, the classification of $2$-dimensional $\frac{1}{q}$-pre-Lie algebras up to isomorphism has been given in \cite{Ka}. According to this classification, when $q\neq -1$, there exist examples of $2$-dimensional simple $q$-pre-Lie algebras.
\item[(5)]
Every Lie algebra $(\mathfrak{g},[\cdot,\cdot])$ admits a compatible $2$-pre-Lie algebra structure given by
\begin{eqnarray}\label{trivial}
    a\circ b=\frac{1}{2}[a,b],\;\;a,b\in\mathfrak{g}.
\end{eqnarray}
We call such compatible $2$-pre-Lie algebra on $(\mathfrak{g},[\cdot,\cdot])$ \bf{trivial}.
\end{itemize}
\end{rmk}

\begin{pro}
    Let $A$ be a vector space with a binary operation $\circ:A\otimes A\to A$. Suppose that \delete{$q\neq 0$ and }$\circ$ is commutative, that is, $a\circ b=b\circ a$ for all $a,b\in A.$ Then $(A,\circ)$ is a $q$-pre-Lie algebra if and only if $(A,\circ)$ is an associative algebra.
\end{pro}
\begin{proof}
    It is clear that Eq. (\ref{q-pre-Lie-2}) holds. Furthermore, we have
    \begin{eqnarray*}
        q(a\circ(b\circ c)-b\circ(a\circ c))-[a,b]\circ c=q((b\circ c)\circ a-b\circ(c\circ a)),\;\;a,b,c\in A.
    \end{eqnarray*}
    So \delete{when $q\neq 0$, }Eq. (\ref{q-pre-Lie-1}) holds if and only if $(A,\circ)$ is an associative algebra. This completes the proof.
\end{proof}
\begin{pro}\label{1/qLcirc}
Let $q\in {\bf k}^{\times}$.
%If $(A, \circ)$ is a $q$-pre-Lie algebra, then $(A, qL_\circ)$ is a representation of the associated Lie algebra $(A, [\cdot,\cdot])$. Moreover, \delete{when $q\neq 0$,}
$(A, \circ)$ is a $q$-pre-Lie algebra if and only if $(A, [\cdot,\cdot])$ is a Lie algebra and $(A, qL_\circ)$ is a representation of $(A, [\cdot,\cdot])$.
\end{pro}
\begin{proof}
By Eq. (\ref{q-pre-Lie-1}), we have $qL_\circ(a)qL_\circ(b)c-qL_\circ(b)qL_\circ(a)c=qL_\circ([a,b])c$. Therefore, $(A, qL_\circ)$ is a representation of the associated Lie algebra $(A, [\cdot,\cdot])$.

\delete{Suppose that $q\neq 0$. }Note that $qL_\circ(a)qL_\circ(b)c-qL_\circ(b)qL_\circ(a)c=qL_\circ([a,b])c$ is equivalent to Eq. (\ref{q-pre-Lie-1}). Moreover, by Eq. (\ref{q-pre-Lie-1}), we obtain
\begin{eqnarray*}
&&q([a,[b,c]]-[[a,b],c]-[b,[a,c]])\\
&=&qa\circ(b\circ c-c\circ b)-q(b\circ c-c\circ b)\circ a-q(a\circ b-b\circ a)\circ c\\
&&+qc\circ (a\circ b-b\circ a)-qb\circ (a\circ c-c\circ a)+q(a\circ c-c\circ a)\circ b\\
&=&(1-q)([a,b]\circ c+[b,c]\circ a+[c,a]\circ b).
\end{eqnarray*}
Then this conclusion follows directly.
\end{proof}
\begin{rmk}
By Proposition \ref{1/qLcirc}, $(A, \circ)$ is a type-$\left(
                                                                         \begin{array}{cc}
                                                                           q & 0 \\
                                                                           0 & q \\
                                                                         \end{array}
                                                                       \right)$
pre-Lie algebra in the sense of \cite{BGLZ}.
\end{rmk}

\begin{defi}
    A bilinear form $\mathcal{B}(\cdot,\cdot)$ on a $q$-pre-Lie algebra $(A,\circ)$ is called {\bf invariant} if
    \begin{eqnarray*}
q\mathcal{B}(a\circ b,c)=-\mathcal{B}(b,[a,c]),\;\; a, b, c\in A.
    \end{eqnarray*}
\end{defi}
\delete{\begin{rmk}
    Note that if a bilinear form $\mathcal{B}(\cdot,\cdot): V\times V\rightarrow {\bf k}$ is $q$-symmetric, then $q=1$ or $-1$. Therefore, if we need to study the relationship between $q$-pre-Lie algebras and $q$-symmetric $2$-cocycles on Lie algebras, it suffices to reduce to the above two special cases. For the cases $q=1$ and $q=-1$, see \cite{Ch} and \cite{LB} respectively.
\end{rmk}}

\delete{\begin{cor}\label{cor-invariant}
    Any $q$-symmetric invariant bilinear form on a $q$-pre-Lie algebra $(A,\circ)$ is a $q$-symmetric $2$-cocycle on the associated Lie algebra $(A,[\cdot,\cdot])$. Conversely, a nondegenerate $q$-symmetric $2$-cocycle on a Lie algebra $(\mathfrak{g},[\cdot,\cdot])$ is invariant on the compatible $q$-pre-Lie algebra given by Eq. (\ref{invariant}).
\end{cor}
\begin{proof}
    Firstly, by Eq. (\ref{invariant}), we get
    \begin{eqnarray*}
        \mathcal{B}([a,b],c)&=&\mathcal{B}(a\circ b,c)-\mathcal{B}(b\circ a,c)=-q\mathcal{B}(b,[a,c])+q\mathcal{B}(a,[b,c])\\
        &=& \mathcal{B}([a,c],b)-\mathcal{B}([b,c],a),\;\;a,b,c\in A.
    \end{eqnarray*}
    Hence $\mathcal{B}(\cdot,\cdot)$ is a $q$-symmetric $2$-cocycle on the associated Lie algebra $(A,[\cdot,\cdot])$. The converse part follows from Theorem \ref{compatible q-pre-Lie}.
\end{proof}}

Recall that two representations $(V_1,\rho_1)$ and $(V_2,\rho_2)$ of a Lie algebra $(\mathfrak{g},[\cdot,\cdot])$ are called {\bf isomorphic} if there exists a linear isomorphism $\varphi:V_1\to V_2$ such that $\varphi(\rho_1(a)v)=\rho_2(a)\varphi(v)$ for all $a\in \mathfrak{g}$ and $v\in V_1$.

\begin{pro}\label{repequi}
    Let $(A,\circ)$ be a finite-dimensional $q$-pre-Lie algebra. Then there is a nondegenerate invariant bilinear form on $(A,\circ)$ if and only if $(A,qL_\circ)$ and $(A^*,{\rm ad}^*)$ are isomorphic as representations of the associated Lie algebra $(A,[\cdot,\cdot])$.
\end{pro}
\begin{proof}
    Suppose that $\varphi:A\to A^*$ is the linear isomorphism satisfying
    \begin{eqnarray*}
        \varphi(qL_\circ(a)b)={\rm ad}^*(a)\varphi(b),\;\;a,b\in A.
    \end{eqnarray*}
    Define a nondegenerate bilinear form $\mathcal{B}(\cdot,\cdot)$ on $A$ by
    \begin{eqnarray}\label{biformphi}
        \mathcal{B}(a,b)=\langle\varphi(a),b\rangle,\;\;a,b\in A.
    \end{eqnarray}
    Then we obtain
    \begin{eqnarray*}
        q\mathcal{B}(a\circ b,c)+\mathcal{B}(b,[a,c])&=&q\langle \varphi(a\circ b),c \rangle+\langle \varphi(b),[a,c] \rangle=q\langle \varphi(L_\circ(a)b),c \rangle+\langle \varphi(b),{\rm ad}(a)c \rangle\\
        &=& \langle {\rm ad}^*(a)\varphi(b),c\rangle+\langle \varphi(b),{\rm ad}(a)c\rangle =0,\;\;a,b,c\in A.
    \end{eqnarray*}
    Hence $\mathcal{B}(\cdot,\cdot)$ is invariant on $(A,\circ)$.

    Conversely, suppose that $\mathcal{B}(\cdot,\cdot)$ is a nondegenerate invariant bilinear form on $(A,\circ)$. Define a linear map $\varphi:A\to A^*$ by Eq. (\ref{biformphi}). Then, by a similar argument, we can prove that $\varphi$ gives the isomorphism between $(A,qL_\circ)$ and $(A^*,{\rm ad}^*)$ as representations of $(A,[\cdot,\cdot])$.
\end{proof}

\delete{Let $(\mathfrak{g}, [\cdot,\cdot])$ be a Lie algebra and $(V, \rho)$ be a representation of $(\mathfrak{g}, [\cdot,\cdot])$. Recall \cite{Ku} that  a linear map $T: V\rightarrow \mathfrak{g}$ is called an {\bf $\mathcal{O}$-operator} on $(\mathfrak{g}, [\cdot,\cdot])$ associated to $(V, \rho)$ if $T$ satisfies
\begin{eqnarray}
[T(u), T(v)]=T(\rho(T(u))v-\rho(T(v))u),\;\;u, v\in V.
\end{eqnarray}

\begin{lem}\label{lem-se1}
Let $(A, \circ)$ be a $q$-pre-Lie algebra. Then $\id$ is an $\mathcal{O}$-operator on the associated Lie algebra $(A, [\cdot,\cdot])$ associated to $(A, qL_\circ)$.
\end{lem}
\begin{proof}
It is straightforward.
\end{proof}

Let $(\mathfrak{g},[\cdot,\cdot])$ be a Lie algebra and $r=\sum_i a_i\otimes b_i\in\mathfrak{g}\otimes \mathfrak{g}$. Recall that $r$ is called a {\bf solution of classical Yang-Baxter equation (CYBE)} in $\mathfrak{g}$ if
    \begin{eqnarray*}
        [r_{12},r_{13}]+[r_{12},r_{23}]+[r_{13},r_{23}]=0
    \end{eqnarray*}
    in the universal enveloping algebra of $\mathfrak{g}$, where
    \begin{eqnarray*}
        r_{12}=\sum_i a_i\otimes b_i\otimes 1,\;\;r_{13}=\sum_i a_i\otimes 1\otimes b_i,\;\;r_{23}=\sum_i 1\otimes a_i\otimes b_i.
    \end{eqnarray*}

Recall that $(V,\rho)$ is a representation of a Lie algebra $(\mathfrak{g},[\cdot,\cdot])$ if and only if there is a Lie algebra structure on the direct sum $\mathfrak{g}\oplus V$ of vector spaces, where the binary operation $[\cdot,\cdot]$ is defined by
\begin{eqnarray*}
    [a+u,b+v]=[a,b]+\rho(a)v-\rho(b)u,\;\;a,b\in\mathfrak{g},\;u,v\in V.
\end{eqnarray*}
We denote it by $\mathfrak{g}\ltimes_\rho V$.

\begin{lem}\label{lem-se2}\cite{Bai}
Let $(\mathfrak{g}, [\cdot,\cdot])$ be a Lie algebra and $(V, \rho)$ be a representation of $(\mathfrak{g}, [\cdot,\cdot])$. Let $T: V\rightarrow A$ be a linear map which is
identified with $r_T\in \mathfrak{g}\otimes V^\ast \subseteq
(\mathfrak{g}\ltimes_{ \rho^\ast} V^\ast) \otimes
(\mathfrak{g}\ltimes_{\rho^\ast} V^\ast)$ through the natural isomorphism
 ${\rm Hom}_{\bf k}(V, \mathfrak{g})\cong \mathfrak{g}\otimes V^\ast$.
Then $r=r_T-\tau r_T$ is a solution of the CYBE in the
$\mathfrak{g}\ltimes_{\rho^\ast} V^\ast$
if and only if $T$ is an
$\mathcal{O}$-operator on $(\mathfrak{g}, [\cdot,\cdot])$ associated to $(V, \rho)$.
\end{lem}

\begin{pro}
    Let $(A,\circ)$ be a $q$-pre-Lie algebra of dimension $n$ and $(A,[\cdot,\cdot])$ the associated Lie algebra of $(A,\circ)$. Then
    \begin{eqnarray*}
        r=\sum_{i=1}^n(e_i\otimes e_i^*-e_i^*\otimes e_i)
    \end{eqnarray*}
    is a solution of CYBE in the Lie algebra $A\ltimes_{qL_\circ^*}A^*$, where $\{e_1,e_2,\ldots,e_n\}$ is a basis of $A$ and $\{e_1^*,e_2^*,\ldots,e_n^*\}$ is the dual basis.
\end{pro}
\begin{proof}
It follows directly from Lemmas \ref{lem-se1} and \ref{lem-se2}.
\end{proof}}

Next, we present some examples of $q$-pre-Lie algebras.
\begin{pro}\label{prop-fg}
    Let $A$ be a vector space with $\dim A\geq2$, and $f,g:A\to {\bf k}$ be two linear functions. Define a binary operation $\circ:A\otimes A\to A$ by
    \begin{eqnarray}\label{fgdef}
        a\circ b:=f(b)a+g(a)b,\;\;a,b\in A.
    \end{eqnarray}
    Then $(A,\circ)$ is a $q$-pre-Lie algebra if and only if $f=0$ or $(q-1)f+g=0$.
\end{pro}
\begin{proof}
    For all $a,b,c\in A$, we get
    \begin{eqnarray*}
        &&(q-1)([a,b]\circ c+[b,c]\circ a+[c,a]\circ b)  \\
        &=&(q-1)(f(c)[a,b]+g([a,b])c+f(a)[b,c]+g([b,c])a+f(b)[c,a]+g([c,a])b)  \\
        &=&(q-1)\big(f(c)(f(b)a+g(a)b-f(a)b-g(b)a)+g(f(b)a+g(a)b-f(a)b-g(b)a)c  \\
        &&+f(a)(f(c)b+g(b)c-f(b)c-g(c)b)+g(f(c)b+g(b)c-f(b)c-g(c)b)a  \\
        &&+f(b)(f(a)c+g(c)a-f(c)a-g(a)c)+g(f(a)c+g(c)a-f(c)a-g(a)c)b \big) \\
        &=&0.
    \end{eqnarray*}
    Then Eq. (\ref{q-pre-Lie-2}) holds. Furthermore, Eq. (\ref{q-pre-Lie-1}) holds if and only if
    \begin{eqnarray}
        \label{Eq-iff}
        ((q-1)f+g)(a)f(c)b+(f(b)g(a)-f(a)g(b))c-((q-1)f+g)(b)f(c)a=0
    \end{eqnarray}
    for all $a,b,c\in A$.

    {\bf Case 1:} Suppose that $\dim A\geq3$. Then for any $a,c\in A$, there exists an element $b\in A$ such that $b$ is linearly independent of $a$ and $c$. So by Eq. (\ref{Eq-iff}), we obtain $((q-1)f+g)(a)f(c)=0$ for all $a,c\in A$. Hence $f=0$ or $(q-1)f+g=0$.

    {\bf Case 2:} Suppose that $\dim A=2$. Let $\{e_1,e_2 \}$ be a basis of $A$. Then Eq. (\ref{Eq-iff}) holds if and only if
    \begin{eqnarray}
        \label{Eq-iff1}
        &((q-1)f+g)(e_1)f(e_1)=0, &\\
        \label{Eq-iff2}
        &((q-1)f+g)(e_2)f(e_2)=0, &\\
        \label{Eq-iff3}
        &-f(e_1)((q-1)f+2g)(e_2)+f(e_2)g(e_1)=0, &\\
        \label{Eq-iff4}
        &-f(e_2)((q-1)f+2g)(e_1)+f(e_1)g(e_2)=0. &
    \end{eqnarray}
    \begin{enumerate}
        \item[(1)] Assume that $f(e_1)\neq0$, then by Eqs. (\ref{Eq-iff1}) and (\ref{Eq-iff4}), we get $((q-1)f+g)(e_1)=0$ and $((q-1)f+g)(e_2)=0$. Hence $(q-1)f+g=0$.
        \item[(2)] Assume that $f(e_1)=0$, then by Eq. (\ref{Eq-iff4}), we get $f(e_2)g(e_1)=0$. If $f(e_2)=0$, then $f=0$. If $f(e_2)\neq0$, then $g(e_1)=0$. By Eq. (\ref{Eq-iff2}) we obtain $((q-1)f+g)(e_2)=0$. Hence $(q-1)f+g=0$.
    \end{enumerate}

    It is clear that Eq. (\ref{Eq-iff}) holds when $f=0$ or $(q-1)f+g=0$. This completes the proof.
\end{proof}
\begin{cor}\label{fgdef-cor}
    With the conditions in Proposition \ref{prop-fg}, we have the following conclusions.
    \begin{enumerate}
        \item[(1)] If $f=0$ and $g\neq0$, then there is a basis $\{e_1,\ldots,e_n \}$ in $A$ such that the nonzero products are given by
        \begin{eqnarray*}
            e_1\circ e_i=e_i,\;\;i=1,\ldots,n.
        \end{eqnarray*}
        \item[(2)] If $f\neq0$ and $(q-1)f+g=0$, then there is a basis $\{e_1,\ldots,e_n \}$ in $A$ such that the only possibly nonzero products are given by
        \begin{eqnarray*}
            e_1\circ e_1=(2-q)e_1,\;\;e_1\circ e_i=(1-q)e_i,\;\;e_i\circ e_1=e_i,\;\;i=2,\ldots,n.
        \end{eqnarray*}
        \item[(3)] If $f=g=0$, then $A$ is trivial.
    \end{enumerate}
\end{cor}
\begin{proof}
    For any nonzero linear function $h:A\to {\bf k}$, choose $e_1\in A$ such that $h(e_1)=1$. Then $A={\bf k}e_1\oplus {\rm Ker}~h$. Taking a basis $\{e_2,\ldots,e_n\}$ of ${\rm Ker}~h$, we obtain a basis $\{e_1,\ldots,e_n\}$ of $A$ such that $h(e_1)=1$ and $h(e_i)=0$ for all $i=2,\ldots,n$. Thus Case (1) follows by taking $g(e_1)=1$ and $g(e_i)=0$ for all $i=2,\ldots,n$ and Case (2) follows by taking $f(e_1)=1$ and $f(e_i)=0$ for all $i=2,\ldots,n$. Case (3) is clear.
\end{proof}

Finally, we present several constructions of $q$-pre-Lie algebras.

For a given $q\in{\bf k}$, recall \cite{Fi} that a {\bf $q$-derivation} of a commutative associative algebra $(A,\cdot)$ is a linear map $\varphi:A\to A$ satisfying
\begin{eqnarray*}
    \varphi(a\cdot b)=q(\varphi(a)\cdot b+a\cdot\varphi(b)),\;\;a,b\in A.
\end{eqnarray*}
\delete{If $q=1$ (resp. $q=-1$), then we get a derivation (resp. anti-derivation) of $(A,\cdot)$.} There is a natural construction of $q$-pre-Lie algebras from commutative associative algebras with a $q$-derivation.\delete{ As a direct consequence of \cite[Example 5]{Ka} and \cite{GD}, we have the following construction of $q$-pre-Lie algebras.}
\begin{pro}\cite[Example 5]{Ka}
    For each $q\in {\bf k}^{\times}$, let $\varphi$ be a $\frac{1}{q}$-derivation of a commutative associative algebra $(A,\cdot)$. Define a binary operation $\circ$ on $A$ by
    \begin{eqnarray*}
        a\circ b:=a\cdot\varphi(b),\;\;a,b\in A.
    \end{eqnarray*}
    Then $(A,\circ)$ is a $q$-pre-Lie algebra.
\end{pro}
\begin{pro}
    Let $\mathcal{B}(\cdot,\cdot)$ be a symmetric bilinear form on a vector space $A$ and $s$ be a fixed element in $A$. For a given $q\in{\bf k}^\times$, define a binary operation $\circ:A\otimes A\to A$ by
    \begin{eqnarray*}
        a\circ b:=\mathcal{B}(a,b)s+q\mathcal{B}(a,s)b,\;\;a,b\in A.
    \end{eqnarray*}
    Then $(A,\circ)$ is a $q$-pre-Lie algebra.
\end{pro}
\begin{proof}
    It is straightforward.
\end{proof}
\begin{pro}\label{daozi}
    Let $P$ be a derivation of a commutative associative algebra $(A,\cdot)$. For a given $q\in{\bf k}^\times$, define a binary operation $\circ$ on $A$ by
    \begin{eqnarray*}
        a\circ b:=a\cdot P(b)+(1-q)P(a)\cdot b+\lambda\cdot a\cdot b,\;\;a,b\in A,
    \end{eqnarray*}
    where $\lambda\in{\bf k}$ or $\lambda\in A$. Then $(A,\circ)$ is a $q$-pre-Lie algebra.
\end{pro}
\begin{proof}
    It is straightforward.
\end{proof}

\delete{By \cite[Theorem 19]{Ka} and \cite{Bu}, we give the classification of $2$-dimensional complex non-commutative $q$-pre-Lie algebras as follows.
\begin{pro}
    Let $(A=\mathbb{C}e_1\oplus\mathbb{C}e_2,\circ)$ be a $2$-dimensional non-commutative $q$-pre-Lie algebra. Then $(A,\circ)$ is isomorphic to one of the following cases:
    \begin{center}
\begin{longtable}{|c|c|c|}
\hline
Type&  the range of $q$ & \makecell{characteristic matrix of $(A,\circ)$} \\
\hline
\endfirsthead
\multicolumn{3}{r}{Continued.}\\
\hline
Type&  the range of $q$  & \makecell{characteristic matrix of $(A,\circ)$} \\
\hline
\endhead
\hline
\endfoot
\hline
\endlastfoot
(Q1)& $q\neq1$ & $\begin{pmatrix}e_1+e_2 & e_2 \\0 &0  \end{pmatrix}$ \\ \hline
(Q2)& $q\neq1$ & $\begin{pmatrix}e_1 & \alpha e_2 \\ 0& 0 \end{pmatrix}$, $\alpha\in{\bf k}$  \\ \hline
(Q3)& $q\neq1$ &$\begin{pmatrix}e_1 & e_2 \\ e_1& e_2 \end{pmatrix}$  \\ \hline
(Q4)&$q\neq1$  &$\begin{pmatrix} e_2& 0 \\0 & 0 \end{pmatrix}$   \\ \hline
(Q5)& $q\neq1$ & $\begin{pmatrix}0 & 0 \\e_1 & 0 \end{pmatrix}$  \\ \hline
(Q6)&$q\neq\frac{1}{2},1$  & $\begin{pmatrix}e_1+e_2 & \frac{q-1}{2q-1}e_2 \\ \frac{q}{2q-1}e_2&0  \end{pmatrix}$  \\ \hline
(Q7)& $q\neq\frac{1}{2},1$ & $\begin{pmatrix}0 & \frac{q-1}{2q-1}e_1 \\ \frac{q}{2q-1}e_1& 0 \end{pmatrix}$  \\ \hline
(Q8)&$q\neq\frac{1}{2},1$  & $\begin{pmatrix}e_1 & \frac{q}{2q-1}e_1+\frac{q-1}{2q-1}e_2 \\ \frac{q-1}{2q-1}e_1+\frac{q}{2q-1}e_2&  e_2\end{pmatrix}$  \\ \hline
(Q9)& $q\neq0,1$ &  $\begin{pmatrix} e_1& \frac{\alpha q-1}{q-1}e_2 \\ \alpha e_2& 0 \end{pmatrix}$, $\alpha\in{\bf k}$ \\ \hline
(Q10)& $q\neq0,1$ & $\begin{pmatrix} e_1&qe_1+(1+q)e_2  \\ (1+q)e_1+qe_2&  e_2\end{pmatrix}$ \\ \hline
(Q11)& $q=\frac{1}{2}$ & $\begin{pmatrix}e_2 & e_2 \\-e_2 &0  \end{pmatrix}$  \\ \hline
(Q12)&$q=\frac{1}{2}$  & $\begin{pmatrix}0 & e_2 \\-e_2 &0  \end{pmatrix}$  \\ \hline
(Q13)& $q=0$ &  $\begin{pmatrix} e_2& 0 \\e_1 &e_2  \end{pmatrix}$ \\ \hline
(Q14)& $q=0$ & $\begin{pmatrix} e_1& e_2 \\e_1-e_2 & 0 \end{pmatrix}$ \\ \hline
(Q15)& $q=0$ & $\begin{pmatrix} e_1& \alpha e_2 \\ \frac{1}{\alpha}e_1& e_2 \end{pmatrix}$, $\alpha\neq0,1$  \\ \hline
(Q16)& $q=0$ & $\begin{pmatrix} e_1& \alpha e_1+\beta e_2 \\ e_1& e_2 \end{pmatrix}$, $\alpha\neq0,\beta\neq1$ or $\alpha\neq1,\beta\neq0$  \\ \hline
(Q17)&$q=0$  & $\begin{pmatrix} e_1&e_1  \\0 &0  \end{pmatrix}$  \\ \hline
(Q18)&$q=0$  & $\begin{pmatrix}e_1 &0  \\ e_1& 0 \end{pmatrix}$  \\ \hline
(Q19)& $q\in{\bf k}$ & $\begin{pmatrix}e_1 &0  \\0 & e_2 \end{pmatrix}$  \\ \hline
(Q20)& $q=1$ & $\begin{pmatrix}e_1 &0  \\ 0& 0 \end{pmatrix}$  \\ \hline
(Q21)&$q=1$  & $\begin{pmatrix}0 & e_1 \\ e_1& e_2 \end{pmatrix}$  \\ \hline
(Q22)& $q=1$ & $\begin{pmatrix}0 & 0 \\0 & e_1 \end{pmatrix}$  \\ \hline
(Q23)& $q=1$ & $\begin{pmatrix}0 & 0 \\-e_1 & \alpha e_2 \end{pmatrix}$, $\alpha\in{\bf k}$  \\ \hline
(Q24)& $q=1$ & $\begin{pmatrix} 0& \alpha e_1 \\ (\alpha-1)e_1& \alpha e_2 \end{pmatrix}$, $\alpha\neq0$  \\ \hline
(Q25)& $q=1$ & $\begin{pmatrix} 0& 0 \\-e_1 &e_1-e_2  \end{pmatrix}$  \\ \hline
(Q26)&$q=1$  & $\begin{pmatrix}e_2 & 0 \\ -e_1& -2e_2 \end{pmatrix}$  \\ \hline
(Q27)& $q=1$ & $\begin{pmatrix}0 &e_1  \\ 0& e_1+e_2 \end{pmatrix}$  \\
\end{longtable}
\end{center}
\vspace{-1cm}
\end{pro}}
\delete{
\begin{defi}
    Let $(A,\circ)$ be a $q$-pre-Lie algebra. A {\bf representation} of $(A,\circ)$ is a triple $(l_\circ,r_\circ,V)$, where $V$ is a vector space, $l_\circ,r_\circ:A\to {\rm End}_{\bf k}(V)$ are two linear maps satisfying
    \begin{eqnarray*}
        &&q(l_\circ(a)l_\circ(b)-l_\circ(b)l_\circ(a))=l_\circ(a\circ b-b\circ a),  \\
        &&q(l_\circ(a)r_\circ(b)-r_\circ(a\circ b))=r_\circ(b)l_\circ(a)-r_\circ(b)r_\circ(a),  \\
        &&(1-q)(r_\circ(a)l_\circ(b)-r_\circ(a)r_\circ(b)+r_\circ(b)r_\circ(a)-r_\circ(b)l_\circ(a))=(1-q)l_\circ(b\circ a-a\circ b),
    \end{eqnarray*}
    for all $a,b\in A$.
\end{defi}
\begin{pro}
    Let $(A,\circ)$ be a $q$-pre-Lie algebra, $V$ be a vector space and $l_\circ,r_\circ:A\to {\rm End}_{\bf k}(V)$ be two linear maps. Then $(l_\circ,r_\circ,V)$ is a representation of $(A,\circ)$ if and only if $A\oplus V$ is a $q$-pre-Lie algebra by defining $\circ_{A\oplus V}$ on $A\oplus V$ by $$(a+u)\circ_{A\oplus V}(b+v)=a\circ b+l_\circ(a)v+r_\circ(b)u ,\;\;a,b\in A,\;u,v\in V.$$ We denote it by $A\ltimes_{l_\circ,r_\circ}V$.
\end{pro}
\begin{proof}
    It is straightforward.
\end{proof}}
\delete{\begin{rmk}
    Let $(A,\circ)$ be a $q$-pre-Lie algebra and $(l_\circ,r_\circ,V)$ a representation of $(A,\circ)$. Then $(A,\circ)$ has a natural dual representation $(q(l_\circ^*-r_\circ^*),-qr_\circ^*,V^*)$ if and only if $q=1$ or $q=-1$.
\end{rmk}}
\subsection{$q$-pre-Lie algebras and $q$-$\mathcal{O}$-operators on Lie algebras}

\begin{defi}
Let $(\mathfrak{g}, [\cdot,\cdot])$ be a Lie algebra and $(V, \rho)$ be a representation of $(\mathfrak{g}, [\cdot,\cdot])$. For each $q\neq0$, a linear map $T: V\rightarrow \mathfrak{g}$ is called a {\bf $q$-$\mathcal{O}$-operator} on $(\mathfrak{g}, [\cdot,\cdot])$ associated to $(V, \rho)$ if $T$ satisfies
\begin{eqnarray}
[T(u), T(v)]=\frac{1}{q}T(\rho(T(u))v-\rho(T(v))u),\;\;u, v\in V.
\end{eqnarray}
A $q$-$\mathcal{O}$-operator is called {\bf strong} if $T$ satisfies
\begin{eqnarray}
(q-1)(\rho([T(u), T(v)])w+\rho([T(v), T(w)])u+\rho([T(w), T(u)])v)=0,\;\;u, v, w\in V.
\end{eqnarray}
\end{defi}
\begin{rmk}
Note that $1$-$\mathcal{O}$-operator is naturally strong and is just the usual $\mathcal{O}$-operator.  $-1$-$\mathcal{O}$-operator is just the anti-$\mathcal{O}$-operator defined in \cite{LB}. Moreover, a $q$-$\mathcal{O}$-operator is just a type-$\left(
                                                                                                          \begin{array}{cc}
                                                                                                            q & 0 \\
                                                                                                            0 & q \\
                                                                                                          \end{array}
                                                                                                        \right)$  $\mathcal{O}$-operator of $(\mathfrak{g}, [\cdot,\cdot])$ associated to $(V, \rho)$ in the sense of \cite{BGLZ}.
\end{rmk}

\begin{pro}\label{pro-constr-Lie}
Let $T: V\rightarrow \mathfrak{g}$ be a $q$-$\mathcal{O}$-operator on a Lie algebra $(\mathfrak{g}, [\cdot,\cdot])$ associated to a representation $(V, \rho)$.
Define a binary operation $\circ: V\otimes V\rightarrow V$ as follows
\begin{eqnarray}\label{eq:q-pre-1}
u\circ v:=\frac{1}{q}\rho(T(u))v,\;\;\;u, v\in V.
\end{eqnarray}
Then $(V, \circ)$ satisfies Eq. (\ref{q-pre-Lie-1}). Moreover, $(V, \circ)$ is Lie-admissible such that $(V, \circ)$ is a $q$-pre-Lie algebra if and only if $T$ is strong. In this case, $T$ is a homomorphism of Lie algebras from the associated Lie algebra $(V,[\cdot,\cdot])$ of $(V, \circ)$ to $(\mathfrak{g}, [\cdot,\cdot])$.
\end{pro}
\begin{proof}
It follows from \cite[Proposition 2.7]{BGLZ} directly.
\delete{Let $u$, $v$, $w\in V$. We obtain
\begin{eqnarray*}
q[u,v]\circ w&=&q^2(\rho(T(u))v-\rho(T(v))u)\circ w\\
&=&q^3(\rho(T(\rho(T(u))v))w-\rho(T(\rho(T(v))u))w)\\
&=&q^2\rho([T(u),T(v)]_{\mathfrak{g}})w\\
&=&q^2(\rho(T(u))\rho(T(v))w-\rho(T(v))\rho(T(u))w)\\
&=&u\circ (v\circ w)-v\circ (u\circ w).
\end{eqnarray*}
Therefore $(V, \circ)$ satisfies Eq. (\ref{q-pre-Lie-1}). Note that
\begin{eqnarray*}
&&(q-1)([u,v]\circ w+[v,w]\circ u+[w,u]\circ v)\\
&&=(q-1)q(\rho(T([u,v]))w+\rho(T([v,w]))u+\rho(T([w,u]))v)\\
&&=(q-1)q(\rho([T(u),T(v)]_{\mathfrak{g}})w+\rho([T(v), T(w)]_{\mathfrak{g}})u+\rho([T(w),T(u)]_{\mathfrak{g}})v).
\end{eqnarray*}
Then this proposition follows directly.}
\end{proof}

\begin{lem}\label{lem-strong}
Let $T: V\rightarrow \mathfrak{g}$ be a $q$-$\mathcal{O}$-operator on a Lie algebra $(\mathfrak{g}, [\cdot,\cdot])$ associated to a representation $(V, \rho)$. If $T$ is invertible, then $T$ is strong.
\end{lem}
\begin{proof}
It follows from \cite[Proposition 2.9]{BGLZ} directly.
\delete{Let $\circ: V\otimes V\rightarrow V$ be the binary operation defined by Eq. (\ref{eq:q-pre-1}). Then we obtain
\begin{eqnarray}
[u,v]=u\circ v-v\circ u=q\rho(T(u))v-q\rho(T(v))u=T^{-1}([T(u), T(v)]_\mathfrak{g}),\;\;u,v\in V.
\end{eqnarray}
 Then it is easy to see that $(V, [\cdot,\cdot])$ is a Lie algebra. Therefore by Proposition \ref{pro-constr-Lie}, $T$ is strong.}
\end{proof}
\begin{pro}\label{pro:comp-q-pre}
Let $(\mathfrak{g}, [\cdot,\cdot])$ be a Lie algebra. Then there is a compatible $q$-pre-Lie algebra structure on $(\mathfrak{g}, [\cdot,\cdot])$ if and only if there exists an invertible $q$-$\mathcal{O}$-operator on $(\mathfrak{g}, [\cdot,\cdot])$.
\end{pro}
\begin{proof}
It follows from \cite[Theorem 2.10]{BGLZ} directly.
\delete{Suppose that there is a compatible  $q$-pre-Lie algebra structure $(\mathfrak{g}, \circ)$ on $(\mathfrak{g}, [\cdot,\cdot])$. Then
\begin{eqnarray*}
[a, b]=a\circ b-b\circ a=q(\frac{1}{q}L_\circ(a)b-\frac{1}{q}L_\circ(b)a),\;\;a, b\in \mathfrak{g}.
\end{eqnarray*}
Therefore $\id: \mathfrak{g}\rightarrow \mathfrak{g}$ is an invertible $q$-$\mathcal{O}$-operator on $(\mathfrak{g}, [\cdot,\cdot])$ associated to $(\mathfrak{g}, \frac{1}{q}L_\circ)$.

Conversely, suppose that $T$ is an invertible $q$-$\mathcal{O}$-operator on $(\mathfrak{g}, [\cdot,\cdot])$ associated to $(V, \rho)$. Then by Lemma \ref{lem-strong},
$T$ is strong. Therefore by Proposition \ref{pro-constr-Lie}, $(V, \circ)$ is a $q$-pre-Lie algebra which is defined by Eq. (\ref{eq:q-pre-1}). Since $T$ is invertible, there is a $q$-pre-Lie algebra structure on $\mathfrak{g}$ given by
\begin{eqnarray*}
a\circ_{\mathfrak{g}} b=T(u)\circ_{\mathfrak{g}} T(v)=T(u\circ v)=qT(\rho(a)T^{-1}(b)),\;\; \text{where $a$, $b\in \mathfrak{g}$, and $T(u)=a$, $T(v)=b$.}
\end{eqnarray*}
Therefore, we obtain
\begin{eqnarray*}
[a,b]&=&[T(u),T(v)]=qT(\rho(T(u))v-\rho(T(v))u)=qT(\rho(a)T^{-1}(b)-\rho(b)T^{-1}(a))\\
&=&a\circ_{\mathfrak{g}} b-b\circ_{\mathfrak{g}} a.
\end{eqnarray*}
Hence $(\mathfrak{g}, \circ_{\mathfrak{g}})$ is a $q$-pre-Lie algebra whose associated Lie algebra is $(\mathfrak{g}, [\cdot,\cdot])$.}
\end{proof}

\delete{
\begin{defi}
Let $V$ be a vector space. A bilinear form $\mathcal{B}(\cdot,\cdot): V\times V\rightarrow {\bf k}$ is called {\bf $q$-symmetric} if
\begin{eqnarray}
\mathcal{B}(u,v)=-q\mathcal{B}(v,u),\;\; u, v\in V.
\end{eqnarray}
\end{defi}}

\delete{\begin{thm}\label{compatible q-pre-Lie}
Let $\mathcal{B}(\cdot,\cdot)$ be a nondegenerate $q$-symmetric $2$-cocycle on a Lie algebra $(\mathfrak{g}, [\cdot,\cdot])$. Then there exists a compatible $q$-pre-Lie algebra structure $\circ$ on $(\mathfrak{g}, [\cdot,\cdot])$ given by
\begin{eqnarray}\label{invariant}
\mathcal{B}(a\circ b,c)=-q\mathcal{B}(b,[a,c]),\;\; a, b, c\in \mathfrak{g}.
\end{eqnarray}
\end{thm}
\begin{proof}
Define a linear map $T:\mathfrak{g}\rightarrow \mathfrak{g}^\ast$ by
\begin{eqnarray}
\langle T(a), b\rangle=\mathcal{B}(a,b),\;\; a, b\in \mathfrak{g}.
\end{eqnarray}
Since $\mathcal{B}(\cdot,\cdot)$ is nondegenerate, $T$ is invertible. Therefore, for any $f$, $g\in \mathfrak{g}^\ast$, there exist $a$, $b\in \mathfrak{g}$ such that
$f=T(a)$, $g=T(b)$. Then for any $c\in \mathfrak{g}$, we obtain
\begin{eqnarray*}
0&=&\mathcal{B}([a,b],c)+\mathcal{B}([b,c],a)+\mathcal{B}([c,a],b)\\
&=&\langle T([a,b]),c\rangle-q\langle T(a),[b,c]\rangle-q\langle T(b), [c,a]\rangle\\
&=&\langle T([T^{-1}(f), T^{-1}(g)]),c\rangle+q\langle\ad^\ast(b)T(a),c\rangle-q\langle\ad^\ast(a)T(b),c\rangle\\
&=&\langle T([T^{-1}(f), T^{-1}(g)]),c\rangle+q\langle\ad^\ast(T^{-1}(g))f,c\rangle-q\langle \ad^\ast(T^{-1}(f))g,c\rangle.
\end{eqnarray*}
Therefore we have $T([T^{-1}(f),T^{-1}(g)])+q\ad^*(T^{-1}(g))f-q\ad^*(T^{-1}(f))g=0$ for any $f$, $g\in \mathfrak{g}^\ast$. Thus $T^{-1}: \mathfrak{g}^\ast\rightarrow \mathfrak{g}$ is an $q$-$\mathcal{O}$-operator on $(\mathfrak{g}, [\cdot,\cdot])$ associated to $(\mathfrak{g}^\ast, \ad^\ast)$. By Proposition \ref{pro:comp-q-pre}, there is a compatible $q$-pre-Lie algebra structure $\circ$ on $(\mathfrak{g}, [\cdot,\cdot])$ given by
\begin{eqnarray*}
a\circ b=qT^{-1}(\ad^\ast(a)T(b)),\;\;a, b\in \mathfrak{g},
\end{eqnarray*}
which gives
\begin{eqnarray*}
\mathcal{B}(a\circ b,c)=\langle T(a\circ b),c\rangle=q\langle \ad^\ast(a)T(b),c\rangle=-q\langle T(b),[a,c]\rangle=-q\mathcal{B}(b, [a,c]),\;\;a, b, c\in \mathfrak{g}.
\end{eqnarray*}
Therefore this theorem holds.
\end{proof}}

\delete{By Proposition \ref{repequi}, Corollary \ref{cor-invariant} and their proofs, we have the following conclusion.
\begin{cor}
    For each $q\neq0$, let $(\mathfrak{g},[\cdot,\cdot])$ be a Lie algebra. If there is a nondegenerate $q$-symmetric $2$-cocycle on $(\mathfrak{g},[\cdot,\cdot])$, then there is a compatible $q$-pre-Lie algebra $(\mathfrak{g},\circ)$ given by Eq. (\ref{invariant}) and furthermore, $(\mathfrak{g},\frac{1}{q}L_\circ)$ and $(\mathfrak{g}^*,{\rm ad}^*)$ are equivalent as representations of $(\mathfrak{g},[\cdot,\cdot])$. Conversely, if there is a compatible $q$-pre-Lie algebra $(\mathfrak{g},\circ)$ such that $(\mathfrak{g},\frac{1}{q}L_\circ)$ and $(\mathfrak{g}^*,{\rm ad}^*)$ are equivalent as representations of $(\mathfrak{g},[\cdot,\cdot])$, then there is a nondegenerate bilinear form $\mathcal{B}(\cdot,\cdot)$ satisfying
    \begin{eqnarray*}
        \mathcal{B}([a,b],c)-q\mathcal{B}(b,[c,a])-q\mathcal{B}(a,[b,c])=0,\;\;a,b,c\in\mathfrak{g}.
    \end{eqnarray*}
\end{cor}}

\subsection{$q$-pre-Lie algebras and $q$-Novikov algebras}
\begin{defi}
Let $A$ be a vector space with a binary operation $\circ: A\otimes A\rightarrow A$ and $q\in {\bf k}^{\times}$. $(A, \circ)$ is called a {\bf $q$-Novikov algebra}, if
\begin{eqnarray}
&&q(a\circ (b\circ c)-b\circ (a\circ c))=[a, b]\circ c,\\
\label{def:q-Nov}&&(q-1)a\circ [b,c]=q((a\circ b)\circ c-(a\circ c)\circ b), \;\;a, b, c\in A,
\end{eqnarray}
where $[a, b]=a\circ b-b\circ a$.
\end{defi}
\begin{rmk}
Note that a $1$-Novikov algebra is just a Novikov algebra and $(-1)$-Novikov algebra is just an admissible Novikov algebra defined in \cite{LB}.
\end{rmk}

\begin{pro}\label{qpretoqnov}
If $q^2-q+1\neq 0$, a $q$-Novikov algebra is a $q$-pre-Lie algebra.
\end{pro}
\begin{proof}
We only need to prove that Eq. (\ref{q-pre-Lie-2}) holds. Let $a$, $b$, $c\in A$.
By Eq. (\ref{def:q-Nov}), we obtain
\begin{eqnarray*}
&&(q-1)(a\circ [b, c]+b\circ [c,a]+c\circ [a,b])\\
&=&q((a\circ b)\circ c-(a\circ c)\circ b+(b\circ c)\circ a-(b\circ a)\circ c+(c\circ a)\circ b-(c\circ b)\circ a)\\
&=&q([a,b]\circ c+[b,c]\circ a+[c,a]\circ b).
\end{eqnarray*}
 Note that
\begin{eqnarray*}
&&q(a\circ [b, c]+b\circ [c,a]+c\circ [a,b])\\
&=&qa\circ (b\circ c-c\circ b)+qb\circ (c\circ a-a\circ c)+qc\circ (a\circ b-b\circ a)\\
&=&q(a\circ (b\circ c)-b\circ (a\circ c))+q(b\circ (c\circ a)-c\circ (b\circ a))+q(c\circ (a\circ b)-a\circ (c\circ b))\\
&=&[a, b]\circ c+[b,c]\circ a+[c,a]\circ b.
\end{eqnarray*}
Therefore, we have $(q^2-(q-1))([a, b]\circ c+[b,c]\circ a+[c,a]\circ b)=0$. Then this conclusion holds.
\end{proof}

\begin{defi}\cite{Dz}
Let $A$ be a vector space with a binary operation $\ast: A\otimes A\rightarrow A$ and $p\in {\bf k}$. Define a binary operation $\circ: A\otimes A\rightarrow A$ by
\begin{eqnarray}\label{p-alg}
a\circ b:=a\ast b+pb\ast a,\;\;a, b\in A.
\end{eqnarray}
Then $(A, \circ)$ is called the {\bf $p$-algebra} of $(A,\ast)$.
\end{defi}
\begin{rmk}
    If $p\neq1$ and $p\neq-1$, by Eq. (\ref{p-alg}), there is an equivalent expression:
    \begin{eqnarray*}
        a*b=\frac{1}{1-p^2}(a\circ b-pb\circ a),\;\;a,b\in A.
    \end{eqnarray*}
\end{rmk}
\begin{pro}\label{qprenov1}
For each $p\neq -1,0, 1$, let $(A, \ast)$ be a pre-Lie algebra and $(A, \circ)$ be the $p$-algebra of $(A, \ast)$. Then $(A, \circ)$ is a $(1-p)$-pre-Lie algebra if and only if $(A, \ast)$ is a Novikov algebra. Moreover, in this case, $(A, \circ)$ is a  $(1-p)$-Novikov algebra.
\end{pro}
\begin{proof}
Denote the associated Lie algebra of $(A, \ast)$ by $(A, \{\cdot,\cdot\})$. Then we obtain
\begin{eqnarray*}
[a,b]=a\circ b-b\circ a=(1-p)(a\ast b-b\ast a)=(1-p)\{a,b\},\;\;a, b\in A.
\end{eqnarray*}
Therefore $(A, \circ)$ is Lie-admissible. Let $a$, $b$, $c\in A$ and $r=1-p$. Then we get
\begin{eqnarray*}
&&r(a\circ (b\circ c)-b\circ (a\circ c))-[a,b]\circ c\\
&=&r\Big( a\ast (b\ast c)+pa\ast(c\ast b)+p(b\ast c)\ast a+p^2(c\ast b)\ast a-b\ast(a\ast c)-pb\ast(c\ast a)\\
&&-p(a\ast c)\ast b-p^2(c\ast a)\ast b\Big) -(1-p)(a\ast b)\ast c-(p-1)(b\ast a)\ast c\\
&&-p(1-p)c\ast (a\ast b)-p(p-1)c\ast (b\ast a)\\
&=&r\Big(p\{a,c\}\ast b+p(b\ast c)\ast a-p\{b,c\}\ast a-p(a\ast c)\ast b+p^2((c\ast b)\ast a-(c\ast a)\ast b)\Big)\\
&=&p(1-p)(1+p)((c\ast b)\ast a-(c\ast a)\ast b).
\end{eqnarray*}
Therefore, if $p\neq -1,0, 1$, then $(A, \circ)$ is a $(1-p)$-pre-Lie algebra if and only if $(A, \ast)$ is a Novikov algebra. Moreover, in this case, we have
\begin{eqnarray*}
&&((1-p)-1)a\circ [b,c]-(1-p)(a\circ b)\circ c+(1-p)(a\circ c)\circ b\\
&=&(1-p)\Big(-pa\ast (b\ast c)+pa\ast(c\ast b)-p^2(b\ast c)\ast a+p^2(c\ast b)\ast a-(a\ast b)\ast c\\
&&-p(b\ast a)\ast c-pc\ast (a\ast b)-p^2c\ast (b\ast a)+(a\ast c)\ast b+p(c\ast a)\ast b+pb\ast (a\ast c)\\
&&+p^2 b\ast (c\ast a)\Big)\\
&=&(1-p)\Big(-p(a\ast b-b\ast a)\ast c+p(a\ast c-c\ast a)\ast b+p^2(b\ast c-c\ast b)\ast a\\
&&-p^2(b\ast c)\ast a+p^2(c\ast b)\ast a-p(b\ast a)\ast c+p(c\ast a)\ast b\Big)\\
&=&0.
\end{eqnarray*}
Therefore $(A, \circ)$ is a  $(1-p)$-Novikov algebra.
\end{proof}
\begin{pro}\label{qprenov2}
    For each $p\neq 0, 1$, let $(A,\circ)$ be a $(1-p)$-pre-Lie algebra and $(A,*)$ be the $(-p)$-algebra of $(A,\circ)$. Then $(A,*)$ is a pre-Lie algebra if and only if $(A,\circ)$ is a $(1-p)$-Novikov algebra. Moreover, in this case, $(A,*)$ is a Novikov algebra.
\end{pro}
\begin{proof}
    Let $a,b,c\in A$. Then we obtain
    \begin{eqnarray*}
        &&(a*b)*c-a*(b*c)-(b*a)*c+b*(a*c)     \\
        &=&(a\circ b-pb\circ a)\circ c-pc\circ(a\circ b-pb\circ a)-a\circ(b\circ c-pc\circ b)+p(b\circ c-pc\circ b)\circ a\\
        &&-(b\circ a-pa\circ b)\circ c+pc\circ(b\circ a-pa\circ b)+b\circ(a\circ c-pc\circ a)-p(a\circ c-pc\circ a)\circ b     \\
        &=&\frac{p^2}{p-1}[a,b]\circ c-(p^2+p)c\circ[a,b]+p[b,c]\circ a+p[c,a]\circ b+(p-p^2)(c\circ b)\circ a\\
        &&+(p^2-p)(c\circ a)\circ b+pa\circ(c\circ b)-pb\circ(c\circ a)    \\
        &=&\frac{p^2}{p-1}[a,b]\circ c-(p^2+p)c\circ[a,b]+p[b,c]\circ a+p[c,a]\circ b +(p-p^2)(c\circ b)\circ a \\
        &&+(p^2-p)(c\circ a)\circ b-\frac{p}{1-p}[b,a]\circ c-pc\circ[b,a]    \\
        &=&p\Big(-pc\circ[a,b]+(1-p)(c\circ b)\circ a+(p-1)(c\circ a)\circ b \Big)   .
    \end{eqnarray*}
    So $(A,*)$ is a pre-Lie algebra if and only if $(A,\circ)$ is a $(1-p)$-Novikov algebra. Furthermore, if $(A,\circ)$ is a $(1-p)$-Novikov algebra, then we get
    \begin{eqnarray*}
        &&(a*b)*c-(a*c)*b\\
        &=& (a\circ b-pb\circ a)\circ c-pc\circ(a\circ b-pb\circ a)-(a\circ c-pc\circ a)\circ b+pb\circ(a\circ c-pc\circ a)    \\
        &=&[a,b]\circ c+(1-p)(b\circ a)\circ c-pc\circ[a,b]-p(1-p)c\circ(b\circ a)-[a,c]\circ b     \\
        &&+(p-1)(c\circ a)\circ b+pb\circ[a,c]+p(1-p)b\circ(c\circ a)     \\
        &=&[a,b]\circ c+[c,a]\circ b+p[b,c]\circ a+pc\circ[b,a]+pb\circ[a,c]+(1-p)(b\circ a)\circ c\\
        &&+(p-1)(c\circ a)\circ b    \\
        &=&(p-1)([b,a]\circ c+[a,c]\circ b)+pa\circ[b,c]+(1-p)(b\circ a)\circ c +(p-1)(c\circ a)\circ b     \\
        &=&(p-1)((a\circ c)\circ b-(a\circ b)\circ c)+pa\circ[b,c]    =0.
    \end{eqnarray*}
    So $(A,*)$ is a Novikov algebra.
\end{proof}
\begin{rmk}
    For each $p$ satisfying $p^2-p+1\neq0$
    and $p\neq -1,0, 1$, Propositions \ref{qprenov1} and \ref{qprenov2} give the following relationship:
    \begin{eqnarray*}
        \{\text{pre-Lie}\}\hookleftarrow\{\text{Novikov}\}\xrightleftharpoons{p\text{-algebra}}{(-p)\text{-algebra}}\{(1-p)\text{-Novikov} \}\hookrightarrow  \{(1-p)\text{-pre-Lie} \}.
    \end{eqnarray*}
\end{rmk}
\begin{ex}
    Let $A$ be a vector space of dimension $n\geq2$ and $f,g:A\to\mathbb{C}$ be two linear functions. Then it is direct to check that Eq. (\ref{fgdef}) defines a Novikov algebra $(A,*)$ if and only if $g=0$, i.e., $a*b=f(b)a$ for all $a,b\in A$. On the other hand, Eq. (\ref{fgdef}) defines a $q$-Novikov algebra $(A,\circ)$ if and only if $(q-1)f+g=0$, i.e., $a\circ b=f(b)a+(1-q)f(a)b$ for all $a,b\in A$, whose classification is given as Cases (2) and (3) in Corollary \ref{fgdef-cor}. Now let $q\neq 1,2$. Clearly, if $(A,\circ)$ is a $q$-Novikov algebra, then the $(q-1)$-algebra $(A,*)$ of $(A,\circ)$ is a Novikov algebra, where
    \begin{eqnarray*}
        a*b=a\circ b+(q-1)b\circ a=(f+(q-1)g)(b)a+(g+(q-1)f)(a)b,\;\;a,b\in A.
    \end{eqnarray*}
    Then we get $g+(q-1)f=0$. Conversely, if $g=(1-q)f$, then $(A,\circ)$ is the $(1-q)$-algebra of the Novikov algebra $(A,*)$ given by $a*b=f(b)a$ for all $a,b\in A$. Hence $(A,\circ)$ is a $q$-Novikov algebra.
\end{ex}

The correspondence linking Novikov algebras to infinite-dimensional Lie algebras (see \cite{BN}) induces a correspondence between $(1-p)$-Novikov algebras and infinite-dimensional Lie algebras, under the restrictions $p\neq0,\pm1$ and $p^2-p+1\neq0$.

\begin{pro}\label{pro-corr}\cite{BN}
    Let $A$ be a vector space with a binary operation $*$. Define a binary operation $[\cdot,\cdot]$ on $A[t,t^{-1}]:=A\otimes{\bf k}[t,t^{-1}] $ by
    \begin{eqnarray*}
        [at^m,bt^n]:=m(a* b)t^{m+n-1}-n(b* a)t^{m+n-1},\;\;a,b\in A,\;m,n\in\mathbb{Z},
    \end{eqnarray*}
    where $at^m:=a\otimes t^m$. Then $(A[t,t^{-1}],[\cdot,\cdot])$ is a Lie algebra if and only if $(A,*)$ is a Novikov algebra.
\end{pro}
\begin{cor}\label{construct1}
    For each $p\neq0,\pm1$ with $p^2-p+1\neq0$, let $A$ be a vector space with a binary operation $\circ$. Define a binary operation $[\cdot,\cdot]$ on $A[t,t^{-1}]:=A\otimes{\bf k}[t,t^{-1}] $ by
    \begin{eqnarray*}
        [at^m,bt^n]:=(m+pn)(a\circ b)t^{m+n-1}-(pm+n)(b\circ a)t^{m+n-1}, \;\;a,b\in A,\; m,n\in\mathbb{Z},
    \end{eqnarray*}
    where $at^m:=a\otimes t^m$. Then $(A[t,t^{-1}],[\cdot,\cdot])$ is a Lie algebra if and only if $(A,\circ)$ is a $(1-p)$-Novikov algebra.
\end{cor}
\begin{proof}
Define a binary operation $*$ on $A$ by
\begin{eqnarray*}
    a*b:=a\circ b-pb\circ a,\;\;a,b\in A.
\end{eqnarray*}
Then we get
\begin{eqnarray*}
    &&m(a*b)t^{m+n-1}-n(b*a)t^{m+n-1}=(m(a\circ b-pb\circ a)-n(b\circ a-pa\circ b))t^{m+n-1}\\
    &=&(m+pn)(a\circ b)t^{m+n-1}-(pm+n)(b\circ a)t^{m+n-1}=[at^m,bt^n],\;\;a,b\in A,\; m,n\in\mathbb{Z}.
\end{eqnarray*}

If $(A,\circ)$ is a $(1-p)$-Novikov algebra, then by Proposition \ref{qpretoqnov}, $(A,\circ)$ is a $(1-p)$-pre-Lie algebra. Hence, by Proposition \ref{qprenov2}, $(A,*)$ is a Novikov algebra. It follows from Proposition \ref{pro-corr} that $(A[t,t^{-1}],[\cdot,\cdot])$ is a Lie algebra.

Conversely, if $(A[t,t^{-1}],[\cdot,\cdot])$ is a Lie algebra, then by Proposition \ref{pro-corr}, $(A,*)$ is a Novikov algebra. Define a binary operation $\diamond$ on $A$ by
\begin{eqnarray*}
    a\diamond b:=\frac{1}{1-p^2}(a*b),\;\;a,b\in A.
\end{eqnarray*}
So $(A,\diamond)$ is a Novikov algebra. It follows from Proposition \ref{qprenov1} that $(A,\circ)$ is a $(1-p)$-Novikov algebra. This completes the proof.
\end{proof}

Next, we provide some constructions of $q$-Novikov algebras from commutative associative algebras.

Recall \cite{LB} that an {\bf admissible pair} on a commutative associative algebra $(A,\cdot)$ is a pair $(P,Q)$, where $P,Q:A\to A$ are two linear maps satisfying
\begin{eqnarray*}
    Q(a\cdot b)=Q(a)\cdot b+a\cdot P(b),\;\;a,b\in A.
\end{eqnarray*}
\begin{pro}\label{construct2}
    For each $p\neq0,\pm1$, let $(P,Q)$ be an admissible pair on a commutative associative algebra $(A,\cdot)$. Define a binary operation $\circ$ on $A$ by
    \begin{eqnarray*}
        a\circ b:=a\cdot Q(b)+pQ(a)\cdot b,\;\;a,b\in A.
    \end{eqnarray*}
    Then $(A,\circ)$ is a $(1-p)$-Novikov algebra. Furthermore, the associated Lie algebra $(A,[\cdot,\cdot])$ of $(A,\circ)$ satisfies
    \begin{eqnarray*}
        [a,b]=(1-p)(a\cdot Q(b)-Q(a)\cdot b)=(1-p)(a\cdot P(b)-P(a)\cdot b),\;\;a,b\in A.
    \end{eqnarray*}
\end{pro}
\begin{proof}
    Define a binary operation $*$ on $A$ by
    \begin{eqnarray*}
        a*b:=a\cdot Q(b),\;\;a,b\in A.
    \end{eqnarray*}
    It follows from \cite[Proposition 3.26]{LB} that $(A,*)$ is a Novikov algebra. By Proposition \ref{qprenov1}, $(A,\circ)$ is a $(1-p)$-Novikov algebra. Moreover, the associated Lie algebra $(A,[\cdot,\cdot])$ of $(A,\circ)$ satisfies
    \begin{eqnarray*}
        [a,b]&=& (a\cdot Q(b)+pQ(a)\cdot b)-(b\cdot Q(a)+pQ(b)\cdot a)=(1-p)(a\cdot Q(b)-Q(a)\cdot b)\\
        &=&(1-p)\big((Q(a\cdot b)-b\cdot P(a))-(Q(a\cdot b)-a\cdot P(b)) \big)=(1-p)(a\cdot P(b)-P(a)\cdot b).
    \end{eqnarray*}
\end{proof}

The construction below is analogous to Proposition \ref{daozi}.
\begin{ex}
    For each $p\neq0$, $p\neq1$ and $p\neq -1$, let $P$ be a derivation on a commutative associative algebra $(A,\cdot)$. Define a binary operation $\circ$ on $A$ by
    \begin{eqnarray*}
        a\circ b=a\cdot P(b)+pP(a)\cdot b+\lambda \cdot a\cdot b,\;\;a,b\in A,
    \end{eqnarray*}
    where $\lambda\in{\bf k}$ or $\lambda\in A$. Then $(A,\circ)$ is a $(1-p)$-Novikov algebra. Note that in this case, $(P,P+\frac{1}{p+1}\lambda{\rm id})(\lambda\in{\bf k})$ and $(P,P+\frac{1}{p+1}L_\cdot(\lambda))(\lambda\in A)$ are admissible pairs.
\end{ex}
\begin{pro}
    For each $p\neq -1, 0,1$ with $p^2-p+1\neq0$, let $(A,\cdot)$ be a commutative associative algebra with an admissible pair $(P,Q)$ and $(V,\circ)$ be a $(1-p)$-Novikov algebra. Define a binary operation $[\cdot,\cdot]$ on $A\otimes V$ by
    \begin{eqnarray*}
        [a\otimes u,b\otimes v]:=(Q(a)\cdot b+pa\cdot Q(b))\otimes u\circ v-(a\cdot Q(b)+pQ(a)\cdot b)\otimes v\circ u,\;a,b\in A,\;u,v\in V.
    \end{eqnarray*}
    Then $(A\otimes V,[\cdot,\cdot])$ is a Lie algebra.
\end{pro}
\begin{proof}
    Define a binary operation $*$ on $V$ by
    \begin{eqnarray*}
        u*v:=u\circ v-pv\circ u,\;\;u,v\in V.
    \end{eqnarray*}
    Since $(V,\circ)$ is a $(1-p)$-Novikov algebra, by Proposition \ref{qpretoqnov}, $(V,\circ)$ is a $(1-p)$-pre-Lie algebra. It follows from Proposition \ref{qprenov2} that $(V,*)$ is a Novikov algebra. Note that
    \begin{eqnarray*}
        &&Q(a)\cdot b\otimes u*v-a\cdot Q(b)\otimes v*u\\
        &=&Q(a)\cdot b\otimes (u\circ v-pv\circ u)-a\cdot Q(b)\otimes (v\circ u-pu\circ v)\\
        &=&(Q(a)\cdot b+pa\cdot Q(b))\otimes u\circ v-(a\cdot Q(b)+pQ(a)\cdot b)\otimes v\circ u\\
        &=&[a\otimes u,b\otimes v],\;\;a,b\in A,\;u,v\in V .
    \end{eqnarray*}
    Therefore, by \cite[Proposition 3.33]{LB}, $(A\otimes V,[\cdot,\cdot])$ is a Lie algebra.
\end{proof}

\begin{rmk}
    Due to Proposition \ref{qpretoqnov}, for each $q\neq0,1,2$ with $q^2-q+1\neq0$, one can obtain some $q$-pre-Lie algebras by Proposition \ref{construct2}.
\end{rmk}
\section{Compatible $q$-pre-Lie algebra structures on the Witt algebra and the Virasoro algebra}
In this section, we classify compatible graded
$q$-pre-Lie algebra structures on the Witt algebra
and investigate a class of compatible $q$-pre-Lie
algebra structures on the Virasoro algebra.

\subsection{Compatible graded $q$-pre-Lie algebra structures on the Witt algebra}

First, we review some results on the Witt algebra and its indecomposable representations. Recall \cite{Ca} that the {\bf Witt algebra} $(\mathcal{W},[\cdot,\cdot])$ is an infinite-dimensional Lie algebra over $\mathbb{C}$ with a basis $\{W_m\mid m\in \mathbb{Z} \}$ satisfying
\begin{eqnarray*}
    [W_m,W_n]=(n-m)W_{m+n},\;\;m,n\in \mathbb{Z}.
\end{eqnarray*}
It is well known that a representation of a Lie algebra $(\mathfrak{g},[\cdot,\cdot])$ is called {\bf indecomposable} if it can not be decomposed into a direct sum of two proper subrepresentations. Recall that a representation $V$ of $\mathcal{W}$ is called a {\bf weight representation} if $V=\bigoplus_{\lambda\in{\mathbb{C}}}V_\lambda$ as vector spaces, where $V_\lambda=\{v\in V\mid W_0v=\lambda v \}$ is called a {\bf weight space}.

Suppose that the vector space $V=\bigoplus_{i\in {\mathbb{Z}}}\mathbb{C}v_i$ is infinite-dimensional. Due to \cite{KS}, the following statements hold.
\begin{enumerate}
    \item[(1)] For any $\alpha\in\mathbb{C}$, $V$ is a representation of $\mathcal{W}$, where the actions are defined as follows:
    \begin{eqnarray*}
        W_mv_i=(m+i)v_{m+i},\;\;m\in\mathbb{Z},i\in\mathbb{Z}^\times\;\;\text{and}\;\; W_mv_0=m(\alpha+m)v_m,\;\;m\in\mathbb{Z}.
    \end{eqnarray*}
    We denote this representation by $V_\alpha$.
    \item[(2)] For any $\beta\in\mathbb{C}$, $V$ is a representation of $\mathcal{W}$, where the actions are defined as follows:
    \begin{eqnarray*}
        W_mv_i=iv_{m+i},\;\;m,i\in\mathbb{Z},m+i\neq0\;\; \text{and}\;\;W_iv_{-i}=-i(\beta+i)v_0,\;\;i\in\mathbb{Z}.
    \end{eqnarray*}
    We denote this representation by $V^\beta$.
    \item[(3)] For $\alpha\in\mathbb{C}$ with $0\leq{\rm Re}\alpha<1$ (where ${\rm Re}\alpha$ denotes the real part of $\alpha$), and for any $\beta\in\mathbb{C}$, $V$ is a representation of $\mathcal{W}$, where the actions are defined as follows:
    \begin{eqnarray*}
        W_mv_i=(\alpha+i+m\beta)v_{m+i},\;\;m,i\in\mathbb{Z}.
    \end{eqnarray*}
    We denote this representation by $V_{\alpha,\beta}$.
\end{enumerate}
\begin{thm}\label{isom-rep}\cite{KS}
    Let $\lambda\in\mathbb{C}$ and $V$ be an indecomposable weight representation of $\mathcal{W}$. Suppose that $V=\bigoplus_{i\in\mathbb{Z}}V_i$, where $V_i=\{v\in V\mid W_0v=(\lambda+i)v \}$ and ${\rm dim}_{\mathbb{C}}V_i=1$ for any $i\in\mathbb{Z}$. Then $V$ is isomorphic to one of $V_\alpha,V^\beta,V_{\alpha,\beta}$ as representations of $\mathcal{W}$.
\end{thm}

Next, we investigate the compatible $q$-pre-Lie algebra structures on $\mathcal{W}$ satisfying
\begin{eqnarray}\label{graded-circ}
    W_m\circ W_n=\varphi(m,n)W_{m+n},\;\;m,n\in\mathbb{Z},
\end{eqnarray}
where $\varphi:\mathbb{Z}\times\mathbb{Z}\to\mathbb{C}$ is a complex-valued function. We denote this $q$-pre-Lie algebra by $(S,\circ)$.
\begin{lem}\label{lem-phi}
    $(S,\circ)$ is a compatible $q$-pre-Lie algebra structure on $\mathcal{W}$ if and only if $\varphi$ satisfies
    \begin{eqnarray}
        \label{phi1}
        &&\varphi(m,n)-\varphi(n,m)=n-m,   \\
        \label{phi2}
        &&(n-m)\varphi(m+n,l)=q(\varphi(n,l)\varphi(m,n+l)-\varphi(m,l)\varphi(n,m+l)),\;\;m,n,l\in\mathbb{Z}.
    \end{eqnarray}
\end{lem}
\begin{proof}
    It is straightforward.
\end{proof}

Based on Eqs. (\ref{phi1}) and (\ref{phi2}), we get the following simple observations.
\begin{enumerate}
    \item[(1)] Let $l=0$ in Eq. (\ref{phi2}). Then we obtain
    \begin{eqnarray}\label{ob1}
        (n-m)\varphi(m+n,0)=q(\varphi(n,0)\varphi(m,n)-\varphi(m,0)\varphi(n,m)),\;\;m,n\in\mathbb{Z}.
    \end{eqnarray}
    \item[(2)] Let $l=1$ in Eq. (\ref{phi2}). Then we obtain
    \begin{eqnarray}\label{ob2}
        (n-m)\varphi(m+n,1)=q(\varphi(n,1)\varphi(m,n+1)-\varphi(m,1)\varphi(n,m+1)),\;\;m,n\in\mathbb{Z}.
    \end{eqnarray}
    \item[(3)] Let $l=2$ in Eq. (\ref{phi2}). Then we obtain
    \begin{eqnarray}\label{ob3}
        (n-m)\varphi(m+n,2)=q(\varphi(n,2)\varphi(m,n+2)-\varphi(m,2)\varphi(n,m+2)),\;\;m,n\in\mathbb{Z}.
    \end{eqnarray}
    \item[(4)] Let $m=0$ in Eq. (\ref{phi2}). Then we obtain
    \begin{eqnarray}\label{ob4}
        (q\varphi(0,n+l)-q\varphi(0,l)-n)\varphi(n,l)=0,\;\;n,l\in\mathbb{Z}.
    \end{eqnarray}
    \item[(5)] Let $m=l=0$ in Eq. (\ref{phi2}). Then we obtain
    \begin{eqnarray}\label{ob5}
        (q\varphi(0,n)-q\varphi(0,0)-n)\varphi(n,0)=0,\;\;n\in\mathbb{Z}.
    \end{eqnarray}
    By Eq. (\ref{ob5}), we get
    \begin{eqnarray*}
        \varphi(n,0)=0\;\;\text{or}\;\;q\varphi(0,0)+n=q\varphi(0,n),\;\;n\in\mathbb{Z}.
    \end{eqnarray*}
    By Eq. (\ref{phi1}), we get
    \begin{eqnarray*}
        \varphi(n,0)=0\;\;\text{or}\;\;q\varphi(0,0)+(1-q)n=q\varphi(n,0),\;\;n\in\mathbb{Z}.
    \end{eqnarray*}
\end{enumerate}

Set
\begin{eqnarray*}
    \Gamma_1=\{m\in\mathbb{Z}\mid \varphi(m,0)=0 \}, \;\;\Gamma_2=\{m\in\mathbb{Z}\mid q\varphi(0,0)+(1-q)m=q\varphi(m,0) \}.
\end{eqnarray*}
Clearly, by Eq. (\ref{ob1}), we obtain $m+n\in\Gamma_1$ for any $m,n\in\Gamma_1$ with $m\neq n$.

Similar to \cite[Corollary 3.4]{BG}, we can obtain the following conclusion.

\begin{pro}\label{formula}
    For each $q\neq1$, we have $q\varphi(0,0)+(1-q)m=q\varphi(m,0)$ for any $m\in\mathbb{Z}$.
\end{pro}

\begin{lem}\label{lem-weight rep}
    For each $q\neq1$, suppose that $(S,\circ)$ is a compatible graded $q$-pre-Lie algebraic structure on $\mathcal{W}$. Define a linear map $\rho:\mathcal{W}\to \mathfrak{gl}(S)$ by
    \begin{eqnarray}\label{maprho}
        \rho(W_m)W_n=qW_m\circ W_n=q\varphi(m,n)W_{m+n},\;\;m,n\in\mathbb{Z}.
    \end{eqnarray}
    Then $(S,\rho)$ is a weight representation of $\mathcal{W}$ whose nonzero weight spaces are all one-dimensional.
\end{lem}
\begin{proof}
    The proof is similar to \cite[Lemma 3.5]{BG}.
\end{proof}

\begin{pro}\label{assump}
    Under the assumptions and notations of Lemma \ref{lem-weight rep}, $(S,\rho)$ is an indecomposable representation of $\mathcal{W}$. Furthermore, $(S,\rho)$ is isomorphic to one of $V_\alpha,V^\beta,V_{\alpha,\beta}$.
\end{pro}
\begin{proof}
    Suppose that $(S,\rho)$ is decomposable, i.e., there exists two nonzero proper subrepresentations $S_1$ and $S_2$ of $S$ such that $S=S_1\oplus S_2$. Since any subrepresentation of a weight representation of $\mathcal{W}$ is also a weight representation, there exists two proper subsets $B_1, B_2$ of $\mathbb{Z}$ such that $B_1\cap B_2=\varnothing$, $B_1\cup B_2=\mathbb{Z}$ and $S_i=\bigoplus_{m\in B_i}\mathbb{C}W_m$ for any $i\in\{1,2\}$. Without loss of generality, we assume that $W_0\in S_1$, i.e., $0\in B_1$.

    Next, we will prove that $\mathbb{Z}_+\subseteq B_1$.

    Suppose that there exists a positive integer in $B_2$. Let $r$ be the minimal positive integer in $B_2$.

    {\bf Claim 1:} $\varphi(m,0)=(1-\frac{1}{q})(r-m)$, $\varphi(0,m)=(1-\frac{1}{q})r+\frac{1}{q}m$ for any $m\in\mathbb{Z}$.

    Since $W_0\in S_1$ and $W_r\in S_2$, we have $\rho(W_r)W_0=q\varphi(r,0)W_r\in S_1\cap S_2=0$ which implies $\varphi(r,0)=0$. By Proposition \ref{formula}, $\varphi(0,0)=(1-\frac{1}{q})r+\varphi(r,0)=(1-\frac{1}{q})r$ and $\varphi(m,0)=\varphi(0,0)+(\frac{1}{q}-1)m=(1-\frac{1}{q})(r-m)$ for any $m\in\mathbb{Z}$. It follows from Eq. (\ref{phi1}) that $\varphi(0,m)=\varphi(m,0)+m=(1-\frac{1}{q})r+\frac{1}{q}m$ for any $m\in\mathbb{Z}$.

    {\bf Claim 2:} Let $m\in B_i$ ($i=1,2$). For any $n\in \mathbb{Z}$, if $\varphi(n,m)\neq0$, then $m+n\in B_i$.

    By Eq. (\ref{maprho}), we obtain $\rho(W_n)W_m=q\varphi(n,m)W_{m+n}\in S_i$. Then $W_{m+n}\in S_i$, which implies $m+n\in B_i$.

    {\bf Claim 3:} Let $m\in B_2$ and $n\in B_1$. Then $m+n\in B_1$ if and only if $\varphi(n,m)=0$, $\varphi(m,n)=n-m$, and $m+n\in B_2$ if and only if $\varphi(m,n)=0$, $\varphi(n,m)=m-n$.

    It follows from Eqs. (\ref{maprho}), (\ref{phi1}) and Claim 2 directly.

    {\bf Claim 4:} $r+1\in B_2$.

    Let $m=-n=-1$, $l=r$ in Eq. (\ref{phi2}). Then we obtain
    \begin{eqnarray}\label{ob6}
        2\varphi(0,r)=q(\varphi(1,r)\varphi(-1,1+r)-\varphi(-1,r)\varphi(1,-1+r)).
    \end{eqnarray}
    Since $\varphi(0,r)=r\neq0$, $\varphi(1,r)\varphi(-1,1+r)\neq\varphi(-1,r)\varphi(1,-1+r)$. Assume $\varphi(-1,r)\neq0$. Note that $W_r\in S_2$ and $\rho(W_{-1})W_r=q\varphi(-1,r)W_{r-1}$, we get $W_{r-1}\in S_2$, i.e., $r-1\in B_2$, which contradicts with the minimality of $r$. Hence $\varphi(1,r)\varphi(-1,1+r)\neq0$, which implies $\varphi(1,r)\neq0$. Thus by Eq. (\ref{maprho}), we have $\rho(W_1)W_r=q\varphi(1,r)W_{r+1}\in S_2$. So $W_{r+1}\in S_2$, i.e., $r+1\in B_2$.

    By Claims 3 and 4, we get $\varphi(r+1,0)=0$. But it follows from Claim 1 that $\varphi(r+1,0)=-(1-\frac{1}{q})\neq0$, which yields a contradiction. Therefore, $B_2$ contains no positive integer, that is, $\mathbb Z_+\subseteq B_1$.

    Similarly, $B_2$ contains no negative integer, and hence $\mathbb Z_-\subseteq B_1$. Together with $0\in B_1$, this implies that $B_1=\mathbb Z$ and $B_2=\varnothing$, contradicting the assumption that $S_2$ is nonzero. Therefore, $(S,\rho)$ is indecomposable. By Theorem \ref{isom-rep}, $(S,\rho)$ is isomorphic to one of $V_\alpha,V^\beta,V_{\alpha,\beta}$.
\end{proof}

\begin{lem}\label{keylem}
    Under the assumptions and notations of Lemma \ref{lem-weight rep}, $(S,\rho)$ is not isomorphic to $V_\alpha$ or $V^\beta$ as $\mathcal{W}$-representations, and if $\beta\neq1-q$, then $(S,\rho)$ is not isomorphic to $V_{\alpha,\beta}$ as $\mathcal{W}$-representations.
\end{lem}
\begin{proof}
    Suppose that $f:S\to V_\alpha$ is a $\mathcal{W}$-representation isomorphism. Set $f(W_0)=\sum_{i\in\mathbb{Z}}a_iv_i$, where $a_i\in\mathbb{C}$ and at most finitely many $a_i$ are nonzero. Then we obtain
    \begin{eqnarray}\label{Valpha}
        \sum_{i\in\mathbb{Z}^\times}ia_iv_i=W_0f(W_0)=f(\rho(W_0)W_0)=q\varphi(0,0)f(W_0)=q\varphi(0,0)\sum_{i\in\mathbb{Z}}a_iv_i.
    \end{eqnarray}

    {\bf Case 1:} If there exists $t\in\mathbb{Z}^\times$ such that $a_t\neq0$, then $\varphi(0,0)=\frac{t}{q}\neq0$. Assume that there exists $i\neq t$ such that $a_i\neq0$, then $\varphi(0,0)=\frac{i}{q}\neq \frac{t}{q}$, which is a contradiction. So $a_i=0$ for any $i\neq t$. Hence $f(W_0)=a_tv_t$. Thus we obtain
    \begin{eqnarray*}
        qtf(W_{-t})=q\varphi(-t,0)f(W_{-t})=f(\rho(W_{-t})W_0)=W_{-t}f(W_0)=W_{-t}(a_tv_t)=0.
    \end{eqnarray*}
    So $f(W_{-t})=0$, which contradicts with that $f$ is a $\mathcal{W}$-representation isomorphism.

    {\bf Case 2:} If $f(W_0)=a_0v_0$ with $a_0\neq0$, then $\varphi(0,0)=0$ due to Eq. (\ref{Valpha}). For any $m\neq0$, we get
    \begin{eqnarray*}
        (1-q)mf(W_m)=q\varphi(m,0)f(W_m)=f(\rho(W_m)W_0)=W_mf(W_0)=W_m(a_0v_0)=m(\alpha+m)a_0v_m.
    \end{eqnarray*}
    So $f(W_m)=\frac{1}{1-q}(\alpha+m)a_0v_m$ for any $m\neq0$. Assume that $m,m+n\in\mathbb{Z}^\times$ and $\alpha+m+n\neq0$, then we have
    \begin{eqnarray*}
        &&\frac{q}{1-q}\varphi(n,m)(\alpha+m+n)a_0v_{m+n}=\frac{q}{1-q}\varphi(n,m)f(W_{m+n})=f(\rho(W_n)W_m)\\
        &=&W_nf(W_m)=\frac{1}{1-q}(\alpha+m)(m+n)a_0v_{m+n}.
    \end{eqnarray*}
    Hence $\varphi(n,m)=\frac{(\alpha+m)(m+n)}{q(\alpha+m+n)}$. Similarly, we get $\varphi(m,n)=\frac{(\alpha+n)(m+n)}{q(\alpha+m+n)}$ for any $n,m+n\in\mathbb{Z}^\times$ with $\alpha+m+n\neq0$. It follows from Eq. (\ref{phi1}) that $m-n=\frac{(m+n)(m-n)}{q(\alpha+m+n)}$ for any $m,n,m+n\in\mathbb{Z}^\times$ with $\alpha+m+n\neq0$, which is a contradiction.

    Taking together the arguments of the above two cases, we show that $(S,\rho)$ is not isomorphic to $V_\alpha$ as $\mathcal{W}$-representations. Similarly, we can prove that $(S,\rho)$ is not isomorphic to $V^\beta$ as $\mathcal{W}$-representations.

    Now suppose that $g:S\to V_{\alpha,\beta}$ is a $\mathcal{W}$-representation isomorphism. Set $g(W_0)=\sum_{i\in\mathbb{Z}}b_iv_i$, where $b_i\in\mathbb{C}$ and at most finitely many $b_i$ are nonzero. Then we obtain
    \begin{eqnarray*}
        q\varphi(0,0)\sum_{i\in\mathbb{Z}}b_iv_i=q\varphi(0,0)g(W_0)=g(\rho(W_0)W_0)=W_0g(W_0)=W_0(\sum_{i\in\mathbb{Z}}b_iv_i)=\sum_{i\in\mathbb{Z}}(\alpha+i)b_iv_i.
    \end{eqnarray*}
    So there exists $t\in\mathbb{Z}$ such that $\varphi(0,0)=\frac{1}{q}(\alpha+t)$ and $g(W_0)=b_tv_t$, where $b_t\neq0$. Let $m\neq0$. Then we have
    \begin{eqnarray*}
        &&(\alpha+t+(1-q)m)g(W_m)=q\varphi(m,0)g(W_m)=g(\rho(W_m)W_0)\\&=&W_mg(W_0)=W_m(b_tv_t)=(\alpha+t+m\beta)b_tv_{m+t}.
    \end{eqnarray*}
    If $\alpha+t+(1-q)m=0$ for some $m\neq0$, then we have $m(\beta-(1-q))b_t=0$, which contradicts with that $\beta\neq 1-q$. So $\alpha+t+(1-q)m\neq0$ for any $m\neq0$. Since $\beta\neq1-q$, $g(W_m)=\frac{\alpha+t+m\beta}{\alpha+t+(1-q)m}b_tv_{m+t}$ for any $m\neq0$. Hence for any $m\neq0$, we get
    \begin{eqnarray*}
        && q\varphi(m,-m)b_tv_t=q\varphi(m,-m)g(W_0)=g(\rho(W_m)W_{-m})=W_mg(W_{-m})\\&=&W_m(\frac{\alpha+t-m\beta}{\alpha+t-(1-q)m}b_tv_{-m+t})=\frac{(\alpha+t-m+m\beta)(\alpha+t-m\beta)}{\alpha+t-(1-q)m}b_tv_t.
    \end{eqnarray*}
    Thus $\varphi(m,-m)=\frac{(\alpha+t-m+m\beta)(\alpha+t-m\beta)}{q(\alpha+t-(1-q)m)}$ for any $m\neq0$. Similarly, we obtain $\varphi(-m,m)=\frac{(\alpha+t+m-m\beta)(\alpha+t+m\beta)}{q(\alpha+t+(1-q)m)}$ for any $m\neq0$. It follows from Eq. (\ref{phi1}) that
    \begin{eqnarray*}
        -2m=\varphi(m,-m)-\varphi(-m,m)=\frac{-2qm(\alpha+t)^2-2(1-q)m^3\beta(\beta-1)}{q(\alpha+t-(1-q)m)(\alpha+t+(1-q)m)},\;\;m\neq0.
    \end{eqnarray*}
    So $\beta^2-\beta+q-q^2=0$. Since $\beta\neq1-q$, we get $\beta=q$. If $q=\frac{1}{2}$, then $\beta=q=1-q$, which yields a contradiction. Hence we may assume $q\neq\frac{1}{2}$ in the remaining computation. This implies $g(W_m)=\frac{\alpha+t+qm}{\alpha+t+(1-q)m}b_tv_{m+t}$ for any $m\neq0$. Let $m\neq0$ and $m+n\neq0$. We have
    \begin{eqnarray*}
        &&q\varphi(n,m)\frac{\alpha+t+qm+qn}{\alpha+t+(1-q)m+(1-q)n}b_tv_{m+n+t}=q\varphi(n,m)g(W_{m+n})=g(\rho(W_n)W_m)\\&=&W_ng(W_m)=W_n(\frac{\alpha+t+qm}{\alpha+t+(1-q)m}b_tv_{m+t})=\frac{(\alpha+t+qm)(\alpha+t+m+qn)}{\alpha+t+(1-q)m}b_tv_{m+n+t}.
    \end{eqnarray*}
    So we get $\varphi(n,m)=\frac{(\alpha+t+qm)(\alpha+t+m+qn)(\alpha+t+(1-q)(m+n))}{q(\alpha+t+(1-q)m)(\alpha+t+q(m+n))}$ for any $m\neq0$ and $m+n\neq0$. Similarly, we obtain $\varphi(m,n)=\frac{(\alpha+t+qn)(\alpha+t+n+qm)(\alpha+t+(1-q)(m+n))}{q(\alpha+t+(1-q)n)(\alpha+t+q(m+n))}$ for any $n\neq0$ and $m+n\neq0$. Thus for any $m\neq0,n\neq0,m+n\neq0,m-n\neq0$, it follows from Eq. (\ref{phi1}) that
    \begin{eqnarray*}
        &&m-n=\varphi(n,m)-\varphi(m,n)\\
        &=&\frac{(m-n)(\alpha+t+(1-q)(m+n))((\alpha+t)^2+q(m+n)(\alpha+t)+(1-q)^2mn)}{(\alpha+t+(1-q)m)(\alpha+t+(1-q)n)(\alpha+t+q(m+n))}.
    \end{eqnarray*}
    Hence $(2q-1)mn(m+n)=0$, which yields a contradiction. So $(S,\rho)$ is not isomorphic to $V_{\alpha,\beta}$ as $\mathcal{W}$-representations when $q\neq1$ and $\beta\neq1-q$.
\end{proof}

\begin{thm}\label{thm-gra-q}
    Let $q\neq1$. Assume that $(S,\circ)$ is a compatible graded  $q$-pre-Lie algebra structure on the Witt algebra $\mathcal{W}$. Then there exists $\lambda\in\mathbb{C}$ such that
    \begin{eqnarray}\label{gra-q}
        W_n\circ W_m=\frac{1}{q}(\lambda+m+(1-q)n)W_{m+n},\;\;m,n\in\mathbb{Z}.
    \end{eqnarray}
    Furthermore, for any $\lambda\in\mathbb{C}$, Eq. (\ref{gra-q}) defines a compatible graded $q$-pre-Lie algebra structure on $\mathcal{W}$, which is denoted by $(S,\circ_\lambda)$.
\end{thm}
\begin{proof}
    It is clear that Eq. (\ref{gra-q}) gives a compatible graded $q$-pre-Lie algebra structure on $\mathcal{W}$. Now assume that $(S,\circ)$ is a compatible graded $q$-pre-Lie algebra structure on $\mathcal{W}$. Then by Proposition \ref{assump} and Lemma \ref{keylem}, $(S,\rho)\cong V_{\alpha,1-q}$ as $\mathcal{W}$-representations, where $\alpha\in\mathbb{C}$ and $0\leq{\rm Re}\alpha<1$. Let $f:S\to V_{\alpha,1-q}$ be the $\mathcal{W}$-representation isomorphism. Set $f(W_0)=\sum_{i\in\mathbb{Z}}a_iv_i$, where $a_i\in\mathbb{C}$ and at most finitely many $a_i$ are nonzero. Then we obtain
    \begin{eqnarray*}
        q\varphi(0,0)\sum_{i\in\mathbb{Z}}a_iv_i=q\varphi(0,0)f(W_0)=f(\rho(W_0)W_0)=W_0(\sum_{i\in\mathbb{Z}}a_iv_i)=\sum_{i\in\mathbb{Z}}(\alpha+i)a_iv_i.
    \end{eqnarray*}
    Hence there exists $t\in\mathbb{Z}$ such that $\varphi(0,0)=\frac{1}{q}(\alpha+t)$ and $f(W_0)=a_tv_t$, where $a_t\neq0$. For any $m\neq0$, we have
\begin{eqnarray}\label{important}
    &&(\alpha+t+(1-q)m)f(W_m)=q\varphi(m,0)f(W_m)=f(\rho(W_m)W_0)\\\notag&=&W_mf(W_0)=W_m(a_tv_t)=(\alpha+t+(1-q)m)a_tv_{m+t}.
\end{eqnarray}

    \delete{{\bf Case 1:} $\alpha\neq0$ or $t$ is not a multiple of $(1-q)$.

    In this case, $\alpha+t+(1-q)m\neq0$ for any $m\neq0$. So $f(W_m)=a_tv_{m+t}$ for all $m\in\mathbb{Z}$. Hence for any $m,n\in\mathbb{Z}$, we get
    \begin{eqnarray*}
        q\varphi(n,m)a_tv_{m+n+t}=f(\rho(W_n)W_m)=W_n(a_tv_{m+t})=(\alpha+t+m+(1-q)n)a_tv_{m+n+t}.
        \end{eqnarray*}
        So $\varphi(n,m)=\frac{1}{q}(\alpha+t+m+(1-q)n)$, which implies $W_n\circ W_m=\frac{1}{q}(\alpha+t+m+(1-q)n)W_{m+n}$ for any $m,n\in\mathbb{Z}$.

        {\bf Case 2:} $\alpha=0$ and $t=0$.

        In this case, by Eq. (\ref{important}), we also get $f(W_m)=a_tv_{m+t}$ for any $m\in\mathbb{Z}$. Similarly, we have $W_n\circ W_m=\frac{1}{q}(\alpha+t+m+(1-q)n)W_{m+n}$ for any $m,n\in\mathbb{Z}$.

        {\bf Case 3:} $\alpha=0$, $t\neq0$ and $t$ is a multiple of $(1-q)$.

        In this case, by Eq. (\ref{important}), we have $f(W_m)=a_tv_{m+t}$ for any $m\neq-\frac{t}{1-q}$. Write $f(W_{-\frac{t}{1-q}})=\sum_{i\in\mathbb{Z}}b_iv_i$, where $b_i\in\mathbb{C}$ and at most finitely many $b_i$ are nonzero. Then
        \begin{eqnarray*}
            \frac{t}{1-q}\sum_{i\in\mathbb{Z}}b_iv_i=q\varphi(0,-\frac{t}{1-q})\sum_{i\in\mathbb{Z}}b_iv_i=f(\rho(W_0)W_{-\frac{t}{1-q}})=W_0(\sum_{i\in\mathbb{Z}}b_iv_i)=\sum_{i\in\mathbb{Z}}ib_iv_i.
        \end{eqnarray*}
        So there exists $c\in\mathbb{C}^\times$ such that $f(W_{-\frac{t}{1-q}})=cv_{\frac{t}{1-q}}$. Hence
        \begin{eqnarray*}
            &&q\varphi(\frac{t}{1-q},-\frac{t}{1-q})a_tv_t=f(\rho(W_{\frac{t}{1-q}})W_{-\frac{t}{1-q}})=W_{\frac{t}{1-q}}f(W_{-\frac{t}{1-q}})=W_{\frac{t}{1-q}}(cv_{\frac{t}{1-q}})=\frac{t(2-q)}{1-q}cv_{\frac{2t}{1-q}},   \\
            &&q\varphi(-\frac{t}{1-q},\frac{t}{1-q})a_tv_t=f(\rho(W_{-\frac{t}{1-q}})W_{\frac{t}{1-q}})=W_{-\frac{t}{1-q}}f(W_{\frac{t}{1-q}})=W_{-\frac{t}{1-q}}(a_tv_{\frac{t(2-q)}{1-q}})=\frac{t}{1-q}a_tv_t.
        \end{eqnarray*}
So $\varphi(-\frac{t}{1-q},\frac{t}{1-q})=\frac{t}{q(1-q)}$.

{\bf Subcase 3.1:} If $q=-1$, then $\varphi(\frac{t}{1-q},-\frac{t}{1-q})=-\frac{3tc}{2a_t}$.

{\bf Subcase 3.2:} If $q\neq-1$, then $\varphi(\frac{t}{1-q},-\frac{t}{1-q})=0$ and $\frac{t(2-q)}{1-q}=0$, which implies $q=2$. Hence $\varphi(-t,t)=0$ and $\varphi(t,-t)=-\frac{t}{2}$. Thus $t=0$, which yields a contradiction.

        {\bf Another method:}}

Now we focus on the set $\mathcal{M}:=\{m\in\mathbb{Z}\mid \alpha+t+(1-q)m=0  \}$. Since $q\neq1$, the equation $\alpha+t+(1-q)m=0$ has at most one solution in $\mathbb{Z}$. So $\mathcal{M}$ is either empty or contains a single integer $m_0$.

        {\bf Case 1:} $\mathcal{M}=\varnothing$.

        In this case, we have $\alpha+t+(1-q)m\neq0$ for any $m\in\mathbb{Z}$. By Eq. (\ref{important}), we have $f(W_m)=a_tv_{m+t}$ for all $m\in\mathbb{Z}$. Hence for any $m,n\in\mathbb{Z}$, we obtain
\begin{eqnarray*}
    &&q\varphi(n,m)a_tv_{m+n+t}=q\varphi(n,m)f(W_{m+n})=f(\rho(W_n)W_m)\\
    &=&W_nf(W_m)=W_n(a_tv_{m+t})=(\alpha+t+m+(1-q)n)a_tv_{m+n+t}.
\end{eqnarray*}
Thus $\varphi(n,m)=\frac{1}{q}(\alpha+t+m+(1-q)n)$, which implies $W_n\circ W_m=\frac{1}{q}(\alpha+t+m+(1-q)n)W_{m+n}$ for any $m,n\in\mathbb{Z}$.

{\bf Case 2:} $\mathcal{M}=\{m_0\}$ for some $m_0\in\mathbb{Z}$.

In this case, we get $m_0=\frac{\alpha+t}{q-1}$. By Eq. (\ref{important}), we obtain $f(W_m)=a_tv_{m+t}$ for any $m\neq m_0$. Note that
\begin{eqnarray*}
    qm_0f(W_{m_0})=(\alpha+t+m_0)f(W_{m_0})=q\varphi(0,m_0)f(W_{m_0})=f(\rho(W_0)W_{m_0})=W_0f(W_{m_0}).
\end{eqnarray*}
So $f(W_{m_0})$ has weight $qm_0$, which equals the weight of $v_{m_0+t}$. Thus there exists $c\in\mathbb{C}^\times$ such that $f(W_{m_0})=cv_{m_0+t}$. Now we consider the following two subcases.

{\bf Subcase 1:} $m_0=0$.

It follows from $f(W_0)=a_tv_t$ that $c=a_t$.

{\bf Subcase 2:} $m_0\neq0$ and $q\neq\frac{1}{2}$.

Note that
\begin{eqnarray*}
    && q\varphi(m_0,-m_0)a_tv_t=q\varphi(m_0,-m_0)f(W_0)=f(\rho(W_{m_0})W_{-m_0})
    =W_{m_0}f(W_{-m_0})\\&=&W_{m_0}(a_tv_{-m_0+t})=(\alpha+t-m_0+m_0(1-q))a_tv_t=-m_0a_tv_t.
\end{eqnarray*}
So $\varphi(m_0,-m_0)=-\frac{m_0}{q}$. Hence $\varphi(-m_0,m_0)=(2-\frac{1}{q})m_0$. Then we have
\begin{eqnarray*}
    &&(2q-1)m_0a_tv_t=f(\rho(W_{-m_0})W_{m_0})=W_{-m_0}(cv_{m_0+t})\\&=&(\alpha+t+m_0-m_0(1-q))cv_t=(2q-1)m_0cv_t.
\end{eqnarray*}
Since $q\neq\frac{1}{2}$, we get $c=a_t$.

{\bf Subcase 3:} $m_0\neq0$ and $q=\frac{1}{2}$.

For any $n\notin\{0,m_0,-m_0 \}$, we have
\begin{eqnarray*}
    &&\frac{1}{2}\varphi(n,m_0)a_tv_{m_0+n+t}  =\frac{1}{2}\varphi(n,m_0)f(W_{m_0+n})=f(\rho(W_n)W_{m_0})=W_nf(W_{m_0})\\&=&W_n(cv_{m_0+t})=(\alpha+t+m_0+\frac{1}{2}n)cv_{m_0+n+t}=\frac{1}{2}(m_0+n)cv_{m_0+n+t},  \\
    &&\frac{1}{2}\varphi(m_0,n)a_tv_{m_0+n+t}    =\frac{1}{2}\varphi(m_0,n)f(W_{m_0+n})=f(\rho(W_{m_0})W_n)=W_{m_0}f(W_n)\\
    &=& W_{m_0}(a_tv_{n+t})=(\alpha+t+n+\frac{1}{2}m_0)a_tv_{m_0+n+t}=na_tv_{m_0+n+t}.
\end{eqnarray*}
Hence $\varphi(n,m_0)=\frac{(m_0+n)c}{a_t}$ and $\varphi(m_0,n)=2n$. It follows from Eq. (\ref{phi1}) that
\begin{eqnarray*}
    m_0-n=\varphi(n,m_0)-\varphi(m_0,n)=\frac{(m_0+n)c}{a_t}-2n.
\end{eqnarray*}
Thus $c=a_t$.

So in this case, we can also show that $W_n\circ W_m=\frac{1}{q}(\alpha+t+m+(1-q)n)W_{m+n}$ for any $m,n\in\mathbb{Z}$ with the same proof as that in Case 1.
\end{proof}
\begin{rmk}
    The compatible graded $1$-pre-Lie algebras, i.e., pre-Lie algebras on the Witt algebra,  have been classified in \cite{KCB}.
\end{rmk}

With a similar proof as that in \cite[Proposition 3.13]{BG}, we can get the following proposition.
\begin{pro}\label{pro-isom-iff}
    Let $q\neq1$, $(S_1,\circ)$ and $(S_2,\bullet)$ be two compatible graded $q$-pre-Lie algebra structures on $\mathcal{W}$ given by
    \begin{eqnarray*}
        W_n\circ W_m=\varphi_1(n,m)W_{m+n},\;\;m,n\in\mathbb{Z},\;\;\text{and}\;\;W_n\bullet W_m=\varphi_2(n,m)W_{m+n},\;\;m,n\in\mathbb{Z},
    \end{eqnarray*}
    respectively, where $\varphi_1,\varphi_2:\mathbb{Z}\times\mathbb{Z}\to\mathbb{C}$ are complex-valued functions. Then $(S_1,\circ)\cong(S_2,\bullet)$ as $q$-pre-Lie algebras if and only if
    \begin{eqnarray*}
        \varphi_1(n,m)=\varphi_2(n,m),\;\;m,n\in\mathbb{Z},\;\;\text{or}\;\;\varphi_1(n,m)=-\varphi_2(-n,-m),\;\;m,n\in\mathbb{Z}.
    \end{eqnarray*}
\end{pro}

Taking  Theorem \ref{thm-gra-q} and Proposition \ref{pro-isom-iff} together, we have the following conclusion.
\begin{thm}
    Under the assumptions and notations of Theorem \ref{thm-gra-q}, for $\lambda_1,\lambda_2\in\mathbb{C}$, $(S,\circ_{\lambda_1})\cong(S,\circ_{\lambda_2})$ as compatible graded $q$-pre-Lie algebra structures on the Witt algebra $\mathcal{W}$ if and only if $\lambda_1=\lambda_2$ or $\lambda_1=-\lambda_2$.
\end{thm}

\subsection{Compatible $q$-pre-Lie algebra structures on the Virasoro algebra}

Recall \cite{Ca} that the {\bf Virasoro algebra} $(\mathcal{V},[\cdot,\cdot])$ is an infinite-dimensional Lie algebra with a basis $\{W_m,{\bf c}\mid m\in\mathbb{Z} \}$ satisfying
\begin{eqnarray}\label{vir-bracket}
    [W_m,W_n]=(n-m)W_{m+n}+\delta_{m+n,0}\frac{m^3-m}{12}{\bf c},\;\;[W_m,{\bf c}]=0,\;\;m,n\in\mathbb{Z}.
\end{eqnarray}
It is well known that the Virasoro algebra $\mathcal{V}$ is the central extension of the Witt algebra $\mathcal{W}$. Here we consider the compatible $q$-pre-Lie algebra structures on $\mathcal{V}$ satisfying
\begin{eqnarray}
    \label{V1}
    &&W_n\circ W_m=\varphi'(n,m)W_{m+n}+\varphi(n,m){\bf c},   \\
    \label{V2}
    &&W_n\circ{\bf c}={\bf c}\circ W_n={\bf c}\circ{\bf c}=0,\;\;m,n\in\mathbb{Z},
\end{eqnarray}
where $\varphi,\varphi':\mathbb{Z}\times\mathbb{Z}\to\mathbb{C}$ are complex-valued functions.

First, we have the following observation.
\begin{pro}
    Let $q\neq1$ and let $(S,\circ)$ be a compatible $q$-pre-Lie algebra structure on $\mathcal{V}$. Then Eq. (\ref{V2}) automatically holds.
\end{pro}
\begin{proof}
    Since $q\neq1$, by Eq. (\ref{q-pre-Lie-2}), we have
    \begin{eqnarray*}
        [W_m,W_n]\circ{\bf c}+[W_n,{\bf c}]\circ W_m+[{\bf c},W_m]\circ W_n=0,\;\;m,n\in\mathbb{Z}.
    \end{eqnarray*}
    It follows from Eq. (\ref{vir-bracket}) that $[W_m,W_n]\circ {\bf c}=0$ and $W_n\circ{\bf c}={\bf c}\circ W_n$ for any $m,n\in\mathbb{Z}$. Then for $n\neq0$, we obtain
    \begin{eqnarray*}
        0=[W_0,W_n]\circ{\bf c}=n(W_n\circ{\bf c}),
    \end{eqnarray*}
    which implies $W_n\circ{\bf c}={\bf c}\circ W_n=0$ for any $n\neq0$. Note that
    \begin{eqnarray*}
        0=[W_1,W_{-1}]\circ{\bf c}=-2(W_0\circ{\bf c}).
    \end{eqnarray*}
    So $W_0\circ{\bf c}={\bf c}\circ W_0=0$. Since
    \begin{eqnarray*}
        0=[W_2,W_{-2}]\circ{\bf c}=-4(W_0\circ{\bf c})+\frac{1}{2}({\bf c}\circ{\bf c}),
    \end{eqnarray*}
    we get ${\bf c}\circ{\bf c}=0$. Therefore, Eq. (\ref{V2}) holds.
\end{proof}

\begin{lem}\label{lem-Vir}
    Let $q\neq1$. If there exists a compatible $q$-pre-Lie algebra structure on $\mathcal{V}$ satisfying Eq. (\ref{V1}), then $\varphi'$ satisfy Eqs. (\ref{phi1}) and (\ref{phi2}). Thus there exists $\lambda\in\mathbb{C}$ such that $\varphi'(n,m)=\frac{1}{q}(\lambda+m+(1-q)n)$ for any $m,n\in\mathbb{Z}$.
\end{lem}
\begin{proof}
   \delete{ Note that
    \begin{eqnarray*}
        &&(m-n)W_{m+n}+\delta_{m+n,0}\frac{n^3-n}{12}{\bf c}=[W_n,W_m]=W_n\circ W_m-W_m\circ W_n\\
        &=&(\varphi'(n,m)W_{m+n}+\varphi(n,m)\delta_{m+n,0}{\bf c})-(\varphi'(m,n)W_{m+n}+\varphi(m,n)\delta_{m+n,0}{\bf c})  \\
        &=&(\varphi'(n,m)-\varphi'(m,n))W_{m+n}+(\varphi(n,m)-\varphi(m,n))\delta_{m+n,0}{\bf c}.
    \end{eqnarray*}
    So $\varphi'$ satisfies Eq. (\ref{phi1}) by comparing the coefficients of $W_{m+n}$. Moreover, by Eq. (\ref{q-pre-Lie-1}), we obtain
    \begin{eqnarray*}
        0&=&q(W_m\circ(W_n\circ W_l)-W_n\circ(W_m\circ W_l))-(W_m\circ W_n)\circ W_l+(W_n\circ W_m)\circ W_l   \\
        &=&q\Big(W_m\circ(\varphi'(n,l)W_{n+l}+\varphi(n,l)\delta_{n+l,0}{\bf c})-W_n\circ(\varphi'(m,l)W_{m+l}+\varphi(m,l)\delta_{m+l,0}{\bf c}) \Big)\\
        &&-(\varphi'(m,n)W_{m+n}+\varphi(m,n)\delta_{m+n,0}{\bf c})\circ W_l+(\varphi'(n,m)W_{m+n}+\varphi(n,m)\delta_{m+n,0}{\bf c})\circ W_l\\
        &=&q\Big(\varphi'(n,l)(\varphi'(m,n+l)W_{m+n+l}+\varphi(m,n+l)\delta_{m+n+l,0}{\bf c})\\
        &&-\varphi'(m,l)(\varphi'(n,m+l)W_{m+n+l}+\varphi(n,m+l)\delta_{m+n+l,0}{\bf c})\Big)\\
        &&-\varphi'(m,n)(\varphi'(m+n,l)W_{m+n+l}+\varphi(m+n,l)\delta_{m+n+l,0}{\bf c})\\
        &&+\varphi'(n,m)(\varphi'(m+n,l)W_{m+n+l}+\varphi(m+n,l)\delta_{m+n+l,0}{\bf c})\\
        &=&\Big(q(\varphi'(n,l)\varphi'(m,n+l)-\varphi'(m,l)\varphi'(n,m+l))+(m-n)\varphi'(m+n,l)  \Big)W_{m+n+l}\\
        &&+\Big(q(\varphi'(n,l)\varphi(m,n+l)-\varphi'(m,l)\varphi(n,m+l))+(m-n)\varphi(m+n,l)  \Big)\delta_{m+n+l,0}{\bf c}.
    \end{eqnarray*}
    Hence $\varphi'$ satisfies Eq. (\ref{phi2}) by observing the coefficient of $W_{m+n+l}$. Thus the result follows from Lemma \ref{lem-phi} and Theorem \ref{thm-gra-q}.} It is straightforward to check that $\varphi'$ satisfies Eqs. (\ref{phi1}) and (\ref{phi2}). Thus the result follows from Lemma \ref{lem-phi} and Theorem \ref{thm-gra-q}.
\end{proof}

In this case, we rewrite Eq. (\ref{V1}) as
\begin{eqnarray}
    \label{rewrite} W_n\circ W_m=\frac{1}{q}(\lambda+m+(1-q)n)W_{m+n}+\varphi(n,m){\bf c},\;\;m,n\in\mathbb{Z}.
\end{eqnarray}

Then we have the following result.
\begin{thm}
    Let $q\neq1$. Then we have the following conclusions.
    \begin{itemize}
        \item[(a)] If $q\neq2$, then there does not exist a compatible $q$-pre-Lie algebra structure on $\mathcal{V}$ satisfying Eq. (\ref{V1}).
        \item[(b)] Assume that $(S,\circ)$ is a compatible  $2$-pre-Lie algebra structure on $\mathcal{V}$ satisfying Eq. (\ref{V1}). Then
        $(S,\circ)$ is just the trivial compatible $2$-pre-Lie algebra on $\mathcal{V}$ given by
        \begin{eqnarray}\label{thm-vir-class}
            W_n\circ W_m=\frac{1}{2}(m-n)W_{m+n}+\frac{n^3-n}{24}  \delta_{m+n,0}{\bf c}=\frac{1}{2}[W_n,W_m],\;\;m,n\in\mathbb{Z}.
        \end{eqnarray}
    \end{itemize}
\end{thm}
\begin{proof}
\begin{itemize}
    \item[(a)]  Suppose that there exists a compatible $q$-pre-Lie algebra structure on $\mathcal{V}$ satisfying Eq. (\ref{V1}). By Lemma \ref{lem-Vir}, Eq. (\ref{rewrite}) holds. By definitions of $q$-pre-Lie algebras and the Virasoro algebra, $\varphi:\mathbb{Z}\times \mathbb{Z}\to\mathbb{C}$ satisfies
\begin{eqnarray}
    \label{V01}&&\varphi(m,-m)-\varphi(-m,m)=\frac{m^3-m}{12},  \\
    \label{V02}&&(\lambda-m-qn)\varphi(m,-m)-(\lambda-n-qm)\varphi(n,-n)=(n-m)\varphi(m+n,-m-n).
\end{eqnarray}
Let $n=0$ in Eq. (\ref{V02}). We have $(\lambda-qm)\varphi(0,0)=\lambda\varphi(m,-m)$.

{\bf Case 1:} $\lambda=0$.

In this case, $\varphi(0,0)=0$. Let $m=-n\neq0$ in Eq. (\ref{V02}). We have
\begin{eqnarray*}
    (q-1)m(\varphi(m,-m)+\varphi(-m,m))=-2m\varphi(0,0)=0,\;\;m\in\mathbb{Z}^\times.
\end{eqnarray*}
So $\varphi(m,-m)+\varphi(-m,m)=0$ for any $m\in\mathbb{Z}^\times$. Hence $\varphi(m,-m)=\frac{m^3-m}{24}$ for any $m\in\mathbb{Z}^\times$. Thus, by Eq. (\ref{V02}), we obtain
\begin{eqnarray*}
    (m^2-n^2)\frac{m^2+n^2+qmn-1}{24}=(m^2-n^2)\frac{m^2+n^2+2mn-1}{24},\;\;m,n\in\mathbb{Z},
\end{eqnarray*}
which contradicts with $q\neq2$.

{\bf Case 2:} $\lambda\neq0$.

In this case, we have $\varphi(m,-m)=\frac{\lambda-qm}{\lambda}\varphi(0,0)$ for any $m\in\mathbb{Z}$. So
\begin{eqnarray*}
    \frac{\lambda-qm}{\lambda}\varphi(0,0)-\frac{\lambda+qm}{\lambda}\varphi(0,0)=\varphi(m,-m)-\varphi(-m,m)=\frac{m^3-m}{12},\;\;m\in\mathbb{Z}.
\end{eqnarray*}
Hence $\varphi(0,0)=-\frac{m^2-1}{24q}\lambda$ for any $m\in\mathbb{Z}^\times$, which yields a contradiction.
\item[(b)] Suppose that there exists a compatible $2$-pre-Lie algebra structure on $\mathcal{V}$ satisfying Eq. (\ref{V1}). In this case, Eqs. (\ref{rewrite}), (\ref{V01}) and (\ref{V02}) still hold.

{\bf Case 1:} $\lambda=0$.

In this case, similar to the calculation in Case 1 of Item (a), we still obtain $\varphi(m,-m)=\frac{m^3-m}{24}$ for any $m\in\mathbb{Z}^\times$. The difference is that this no longer yields a contradiction. \delete{If $m=-n$, then by Eq. (\ref{rewrite}), we have
\begin{eqnarray*}
    W_n\circ W_m&=&\frac{1}{2}(m-n)W_{m+n}+\varphi(n,m){\bf c}=\frac{1}{2}(m-n)W_{m+n}+\varphi(n,-n){\bf c}\\&=&\frac{1}{2}(m-n)W_{m+n}+\frac{n^3-n}{24}{\bf c}.
\end{eqnarray*}
By Eq. (\ref{vir-bracket}), we get Eq. (\ref{thm-vir-class}).} If $m\neq-n$, then by definitions of $2$-pre-Lie algebras and the Virasoro algebra, $\varphi:\mathbb{Z}\times \mathbb{Z}\to\mathbb{C}$ satisfies
\begin{eqnarray}
    \label{b-case1-1}&&\varphi(m,n)=\varphi(n,m),\\
    \label{b-case1-2}&&(n-m)\varphi(0,m+n)=(m+n)\varphi(m,n).
\end{eqnarray}
Let $m\neq0,n=0$ in Eq. (\ref{b-case1-2}). We get $\varphi(0,m)=\varphi(m,0)=0$ for any $m\in\mathbb{Z}^\times$. So for any $m,n\in\mathbb{Z}$ with $m\neq-n$, $\varphi(0,m+n)=0$. Then it follows from Eq. (\ref{b-case1-2}) that $\varphi(m,n)=0$ when $m\neq-n$. Therefore, we get $\varphi(m,n)=\frac{m^3-m}{24}\delta_{m+n,0}$ for any $m$, $n\in \mathbb{Z}$. Then the result follows by Eq. (\ref{rewrite}). \delete{Thus by Eq. (\ref{rewrite}), we have
\begin{eqnarray*}
    W_n\circ W_m=\frac{1}{2}(m-n)W_{m+n}.
\end{eqnarray*}
By Eq. (\ref{vir-bracket}), we obtain Eq. (\ref{thm-vir-class}).}

{\bf Case 2:} $\lambda\neq0$.

The proof for this case is the same as Case 2 in Item (a), and we still get a contradiction. Therefore, this case can't exist.
\end{itemize}

This completes the proof.\delete{The proof follows by an argument similar to that of \cite[Theorem 4.2]{BG}.}
\end{proof}
\begin{rmk}
    By the results in \cite{KCB}, there exist compatible graded $1$-pre-Lie algebras, i.e., pre-Lie algebras on the Virasoro algebra.
\end{rmk}

\section{Compatible $q$-pre-Lie algebra structures on finite-dimensional complex simple Lie algebras}\label{sec:compatible-simple}

In this section, we classify compatible $q$-pre-Lie algebra structures on any
finite-dimensional complex simple Lie algebra $\mathfrak{g}$.  First, we prove that compatible $q$-pre-Lie algebra structure on complex simple Lie algebras $\mathfrak{g}$ not isomorphic to $\mathfrak{sl}_2(\mathbb C)$ exists if and only if $q=2$, in which case it is unique and is given by $a\circ b=\frac12[a,b]$ for any $a,b\in\mathfrak g$. For $\mathfrak{sl}_2(\mathbb C)$, such structures exist if and only if $q=2$ or $q=-1$. For the case $q=2$, the structure is given by $a\circ b=\frac12[a,b]$ for any $a,b\in\mathfrak{sl}_2(\mathbb C)$. For the case $q=-1$, we give an explicit classification of all such structures in terms of invertible $(-1)$-derivations and obtain an isomorphism criterion for these structures.

\subsection{Compatible $q$-pre-Lie algebra structures on finite-dimensional complex simple Lie algebras other than $\mathfrak{sl}_2(\mathbb C)$}

First, we recall that a {\bf commutative $2$-cocycle} $\mathcal{B}(\cdot,\cdot)$ on a Lie algebra $(\mathfrak{g},[\cdot,\cdot])$ is a symmetric bilinear form satisfying
\begin{eqnarray}\label{eq:simple-comm2}
    \mathcal{B}([a,b],c)+\mathcal{B}([b,c],a)+\mathcal{B}([c,a],b)=0,\;\;a,b,c\in\mathfrak{g}.
\end{eqnarray}

\delete{\begin{lem}\label{lem:no-compatible-pre-Lie}
    A nonzero finite-dimensional complex semisimple Lie algebra admits no compatible $1$-pre-Lie algebra structure.
\end{lem}
\begin{proof}
    Suppose that $(\mathfrak{g},\circ)$ is a compatible $1$-pre-Lie algebra structure on a finite-dimensional complex semisimple Lie algebra $(\mathfrak{g},[\cdot,\cdot])$. By Proposition \ref{1/qLcirc}, $(\mathfrak{g},L_\circ)$ is a representation of $(\mathfrak{g},[\cdot,\cdot])$. It is clear that $\mathrm{id}_\mathfrak{g}$ is a $1$-cocycle for $(\mathfrak{g},L_\circ)$. By Whitehead's first lemma, there exists $v\in\mathfrak{g}$ such that $a=L_\circ(a)v=a\circ v$ for all $a\in\mathfrak{g}$. So $R_\circ(v)=\mathrm{id}_\mathfrak{g}$.

    On the other hand, we have $\mathfrak g=[\mathfrak g,\mathfrak g]$, which implies $\mathrm{tr}\,\mathrm{ad}(a)=0$ for all $a\in\mathfrak{g}$. Since $(\mathfrak{g},L_\circ)$ is a representation of $(\mathfrak{g},[\cdot,\cdot])$, $\mathrm{tr}\,L_\circ(a)=0$ for all $a\in \mathfrak{g}$. Then for any $a\in \mathfrak{g}$, it follows from $\mathrm{ad}(a)=L_\circ(a)-R_\circ(a)$ that $\mathrm{tr}\,R_\circ(a)=0$. Thus $0=\mathrm{tr}\,R_\circ(v)=\mathrm{dim}\,\mathfrak{g}$, which yields a contradiction.
\end{proof}}

\begin{thm}\label{thm:classification-not-sl2}
    Let $(\mathfrak g,[\cdot,\cdot])$ be a finite-dimensional complex simple Lie algebra not isomorphic to $\mathfrak{sl}_2(\mathbb C)$. Then there exists a compatible $q$-pre-Lie algebra structure $(\mathfrak{g},\circ)$ on $(\mathfrak g,[\cdot,\cdot])$ if and only if $q=2$. Moreover, such a structure is unique and is given by Eq. (\ref{trivial}).
\end{thm}
\begin{proof}
By \cite[Corollary 21]{He}, a finite-dimensional complex semisimple Lie algebra admits no compatible $1$-pre-Lie algebra structure. Therefore, we only need to consider the case $q\neq1$. Suppose that $(\mathfrak{g},\circ)$ is a compatible $q$-pre-Lie algebra structure on $(\mathfrak g,[\cdot,\cdot])$. Write
    \begin{eqnarray}\label{eq:symmetric-decomposition}
        a\circ b=\frac12[a,b]+S(a,b),\;\; a,b\in\mathfrak g,
    \end{eqnarray}
    where $S:\mathfrak{g}\times\mathfrak{g}\to\mathfrak{g}$ is a symmetric bilinear map. For all $a,b,c\in\mathfrak{g}$, it follows from Eqs. (\ref{q-pre-Lie-2}), (\ref{eq:symmetric-decomposition}) and the Jacobi identity that
    \begin{eqnarray*}
        &&S([a,b],c)+S([b,c],a)+S([c,a],b)\\
        &=&[a,b]\circ c+[b,c]\circ a+[c,a]\circ b-\frac{1}{2}([[a,b],c]+[[b,c],a]+[[c,a],b])\\
        &=&0.
    \end{eqnarray*}
    For any $f\in\mathfrak g^*$, define a symmetric bilinear form $\mathcal{B}_f(\cdot,\cdot)$ on $(\mathfrak{g},[\cdot,\cdot])$ by $$\mathcal{B}_f(a,b)=f(S(a,b)),\;\;a,b\in\mathfrak{g}.$$ Then for all $a,b,c\in\mathfrak{g}$, we get
    \begin{eqnarray*}
        \mathcal{B}_f([a,b],c)+\mathcal{B}_f([b,c],a)+\mathcal{B}_f([c,a],b)=f(S([a,b],c)+S([b,c],a)+S([c,a],b))=0.
    \end{eqnarray*}
    Since $\mathfrak g$ is not isomorphic to $\mathfrak{sl}_2(\mathbb C)$, by \cite[Corollary 2.5]{DZ}, $\mathcal{B}_f(a,b)=0$ for all $a,b\in\mathfrak{g}$. So $S(a,b)=0$ for all $a,b\in\mathfrak{g}$, which implies $a\circ b=\frac{1}{2}[a,b]$ for all $a,b\in\mathfrak{g}$. Hence $L_\circ(a)=\frac{1}{2}\mathrm{ad}(a)$ for any $a\in\mathfrak{g}$. By Proposition \ref{1/qLcirc}, for all $a,b\in\mathfrak{g}$,
    \begin{eqnarray*}
        0=qL_\circ([a,b])-[qL_\circ(a),qL_\circ(b)]=\frac{q}{2}\mathrm{ad}([a,b])-\frac{q^2}{4}[\mathrm{ad}(a),\mathrm{ad}(b)]=(\frac{q}{2}-\frac{q^2}{4})\mathrm{ad}([a,b]).
    \end{eqnarray*}
    Since $q\neq0$ and $\mathfrak{g}$ is non-abelian and simple, we conclude that $q=2$. Conversely, every Lie algebra $(\mathfrak{g},[\cdot,\cdot])$ admits a compatible $2$-pre-Lie algebra structure $(\mathfrak{g},\circ)$ given by Eq. (\ref{trivial}). This completes the proof.
\end{proof}

\subsection{Compatible $q$-pre-Lie algebra structures on $\mathfrak{sl}_2(\mathbb C)$}

Recall that $\mathfrak{sl}_2(\mathbb{C})=\mathrm{span}_\mathbb{C}\{e=\begin{pmatrix} 0 &1 \\0&0 \end{pmatrix},f=\begin{pmatrix}  0&0\\1&0 \end{pmatrix},h=\begin{pmatrix} 1&0\\0&-1 \end{pmatrix}  \}$ is a $3$-dimensional simple Lie algebra satisfying
\begin{eqnarray*}
    [h,e]=2e,\;\;[h,f]=-2f,\;\;[e,f]=h.
\end{eqnarray*}
\begin{defi}
    For a Lie algebra $(\mathfrak{g},[\cdot,\cdot])$ and $q\in\mathbb{C}$, a linear map $\varphi:\mathfrak{g}\to\mathfrak{g}$ is called a {\bf $q$-derivation} of $(\mathfrak{g},[\cdot,\cdot])$ if
    \begin{eqnarray}\label{Lie-q-Der}
        \varphi([a,b])=q([\varphi(a),b]+[a,\varphi(b)]),\;\;a,b\in\mathfrak{g}.
    \end{eqnarray}
\end{defi}

\begin{pro}\label{prop:sl2-existence-parametrization}
    For any $q\in\mathbb C^\times$, there exists a compatible $q$-pre-Lie algebra structure on $\mathfrak{sl}_2(\mathbb C)$ if and only if there exists an invertible $\frac1q$-derivation of $\mathfrak{sl}_2(\mathbb C)$. Moreover, all such structures are given by
    \begin{eqnarray}\label{eq:multiplication-varphi}
        a\circ_\varphi b=\frac1q\varphi^{-1}([a,\varphi(b)]),
        \;\; a,b\in\mathfrak{sl}_2(\mathbb C),
    \end{eqnarray}
    where $\varphi$ is an invertible $\frac1q$-derivation of $\mathfrak{sl}_2(\mathbb C)$.
\end{pro}
\begin{proof}
    Suppose that $(\mathfrak{sl}_2(\mathbb C),\circ)$ is a compatible $q$-pre-Lie algebra structure on $(\mathfrak{sl}_2(\mathbb C),[\cdot,\cdot])$. By Proposition \ref{1/qLcirc}, $(\mathfrak{sl}_2(\mathbb C),qL_\circ)$ is a $3$-dimensional representation of $(\mathfrak{sl}_2(\mathbb C),[\cdot,\cdot])$. We assume that $(\mathfrak{sl}_2(\mathbb C),qL_\circ)$ is reducible. Then we have an $\mathfrak{sl}_2(\mathbb C)$-module isomorphism $\mathfrak{sl}_2(\mathbb C)\cong V_1\oplus V_0$ or $\mathfrak{sl}_2(\mathbb C)\cong V_0\oplus V_0\oplus V_0$, where $V_n$ is the irreducible module of highest weight $n$. In the first case, identify $V_0$ with $\mathbb C$, choose an $\mathfrak{sl}_2(\mathbb C)$-module isomorphism $f:V_1\oplus \mathbb{C}\to\mathfrak{sl}_2(\mathbb C)$ and define a linear map $g:\mathfrak{sl}_2(\mathbb C)\to\mathbb C$ by
    $$g(f(v,z))=z,\;\;v\in V_1,\;z\in\mathbb{C}.$$
    In the second case, identify $V_0\oplus V_0\oplus V_0$ with $\mathbb{C}\oplus\mathbb{C}\oplus\mathbb{C}$, choose an $\mathfrak{sl}_2(\mathbb C)$-module isomorphism $f: \mathbb{C}\oplus\mathbb{C}\oplus\mathbb{C}\to\mathfrak{sl}_2(\mathbb C)$ and define a linear map $g:\mathfrak{sl}_2(\mathbb C)\to\mathbb C$ by
    $$g(f(z_1,z_2,z_3))=z_1,\;\;z_1,z_2,z_3\in\mathbb{C}.$$
    In both cases, $g$ is a nonzero $\mathfrak{sl}_2(\mathbb C)$-module homomorphism, where $\mathbb C$ is regarded as the trivial $\mathfrak{sl}_2(\mathbb C)$-module. Note that
    \begin{eqnarray}\label{apply-phi}
        q[a,b]=q(a\circ b-b\circ a)=qL_\circ(a)b-qL_\circ(b)a,\;\;a,b\in\mathfrak{sl}_2(\mathbb{C}).
    \end{eqnarray}
    So $qg([a,b])=0$ for all $a,b\in\mathfrak{sl}_2(\mathbb{C})$. Since $q\neq0$, we conclude that $g([a,b])=0$ for all $a,b\in\mathfrak{sl}_2(\mathbb{C})$. It follows from $[\mathfrak{sl}_2(\mathbb{C}),\mathfrak{sl}_2(\mathbb{C})]=\mathfrak{sl}_2(\mathbb{C})$ that $g=0$, which yields a contradiction. Hence $(\mathfrak{sl}_2(\mathbb C),qL_\circ)$ is irreducible. Since every $3$-dimensional irreducible $\mathfrak{sl}_2(\mathbb C)$-module is isomorphic to the adjoint module, there exists an invertible linear map $\varphi$ on $\mathfrak{sl}_2(\mathbb C)$ such that $$\varphi(qL_\circ(a)b)=[a,\varphi(b)],
\;\; a,b\in\mathfrak{sl}_2(\mathbb C).$$
It follows that $$a\circ b=\frac{1}{q}\varphi^{-1}([a,\varphi(b)]),
\;\; a,b\in\mathfrak{sl}_2(\mathbb C).$$
By Eq. (\ref{apply-phi}), we obtain $$q\varphi([a,b])=\varphi(qL_\circ(a)b)-\varphi(qL_\circ(b)a)=[a,\varphi(b)]+[\varphi(a),b] ,\;\;  a,b\in\mathfrak{sl}_2(\mathbb C),$$
which implies $\varphi$ is an invertible $\frac1q$-derivation of $\mathfrak{sl}_2(\mathbb C)$.

Conversely, suppose that $\varphi$ is an invertible $\frac1q$-derivation of $\mathfrak{sl}_2(\mathbb C)$. Define a binary operation $\circ_\varphi:\mathfrak{sl}_2(\mathbb C)\otimes\mathfrak{sl}_2(\mathbb C)\to\mathfrak{sl}_2(\mathbb C)$ by Eq. (\ref{eq:multiplication-varphi}). Then for all $a,b\in\mathfrak{sl}_2(\mathbb C)$, we have
\begin{eqnarray*}
    &&qL_{\circ_\varphi}([a,b])-[qL_{\circ_\varphi}(a),qL_{\circ_\varphi}(b)]\\
    &=&\varphi^{-1}\mathrm{ad}([a,b])\varphi-[\varphi^{-1}\mathrm{ad}(a)\varphi,\varphi^{-1}\mathrm{ad}(b)\varphi ]  \\
    &=&\varphi^{-1}\mathrm{ad}([a,b])\varphi-\varphi^{-1}\mathrm{ad}(a)\mathrm{ad}(b)\varphi+\varphi^{-1}\mathrm{ad}(b)\mathrm{ad}(a)\varphi\\
    &=&\varphi^{-1}(\mathrm{ad}([a,b])-[\mathrm{ad}(a),\mathrm{ad}(b)])\varphi\\
    &=&0,
\end{eqnarray*}
which implies $(\mathfrak{sl}_2(\mathbb C),qL_{\circ_\varphi})$ is a representation of $(\mathfrak{sl}_2(\mathbb C),[\cdot,\cdot])$. By Proposition \ref{1/qLcirc}, $(\mathfrak{sl}_2(\mathbb C),\circ_\varphi)$ is a $q$-pre-Lie algebra. Furthermore, since $\varphi$ is a $\frac1q$-derivation, we have
\begin{eqnarray*}
    \varphi(a\circ_\varphi b-b\circ_\varphi a)=\frac{1}{q}([a,\varphi(b)]-[b,\varphi(a)])=\frac{1}{q}([a,\varphi(b)]+[\varphi(a),b])=\varphi([a,b]),\;\;a,b\in\mathfrak{sl}_2(\mathbb{C}).
\end{eqnarray*}
Since $\varphi$ is invertible, it follows that $a\circ_\varphi b-b\circ_\varphi a=[a,b]$ for any $a,b\in\mathfrak{sl}_2(\mathbb{C})$. Therefore, $(\mathfrak{sl}_2(\mathbb C),\circ_\varphi)$ is a compatible $q$-pre-Lie algebra structure on $(\mathfrak{sl}_2(\mathbb C),[\cdot,\cdot])$. This completes the proof.
\end{proof}

\begin{pro}\label{prop:sl2-same-structure}
    For any $q\in\mathbb C^\times$, let $\varphi_1$ and $\varphi_2$ be two invertible $\frac1q$-derivations of $\mathfrak{sl}_2(\mathbb C)$. Then the compatible $q$-pre-Lie algebra structures $(\mathfrak{sl}_2(\mathbb C),\circ_{\varphi_1})$ and $(\mathfrak{sl}_2(\mathbb C),\circ_{\varphi_2})$ given by Eq. (\ref{eq:multiplication-varphi}) coincide if and only if $\varphi_2=\lambda\varphi_1$ for some $\lambda\in\mathbb C^\times$.
\end{pro}
\begin{proof}
    If $(\mathfrak{sl}_2(\mathbb C),\circ_{\varphi_1})$ and $(\mathfrak{sl}_2(\mathbb C),\circ_{\varphi_2})$ coincide, then by Eq. (\ref{eq:multiplication-varphi}), we have
    \begin{eqnarray*}
        \varphi_1^{-1}\mathrm{ad}(a)\varphi_1=qL_{\circ_{\varphi_1}}(a)=qL_{\circ_{\varphi_2}}(a)=\varphi_2^{-1}\mathrm{ad}(a)\varphi_2,\;\;a\in\mathfrak{sl}_2(\mathbb C).
    \end{eqnarray*}
    So $\varphi_2\varphi_1^{-1}$ commutes with $\mathrm{ad}(a)$ for any $a\in\mathfrak{sl}_2(\mathbb C)$. Since the adjoint module of $\mathfrak{sl}_2(\mathbb C)$ is irreducible, Schur's lemma implies that $\varphi_2\varphi_1^{-1}=\lambda\mathrm{id}_{\mathfrak{sl}_2(\mathbb C)}$ for some $\lambda\in\mathbb C^\times$, that is, $\varphi_2=\lambda\varphi_1$.

    Conversely, if $\varphi_2=\lambda\varphi_1$ for some $\lambda\in\mathbb C^\times$, then by Eq. (\ref{eq:multiplication-varphi}), we get
    \begin{eqnarray*}
        a\circ_{\varphi_2}b=\frac{1}{q}\varphi_2^{-1}([a,\varphi_2(b)])
=\frac{1}{q}(\lambda\varphi_1)^{-1}([a,\lambda\varphi_1(b)])
=\frac{1}{q}\varphi_1^{-1}([a,\varphi_1(b)])=a\circ_{\varphi_1}b,\;\;a,b\in\mathfrak{sl}_2(\mathbb C).
    \end{eqnarray*}
    So $(\mathfrak{sl}_2(\mathbb C),\circ_{\varphi_1})$ and $(\mathfrak{sl}_2(\mathbb C),\circ_{\varphi_2})$ coincide. This completes the proof.
\end{proof}

\begin{lem}\label{lem:q-derivation-equations}
    Let $\varphi:\mathfrak{sl}_2(\mathbb C)\to\mathfrak{sl}_2(\mathbb C)$ be a linear map. Then $\varphi$ is a $\frac1q$-derivation of $(\mathfrak{sl}_2(\mathbb C),[\cdot,\cdot])$ if and only if $\varphi$ is given by
    \begin{eqnarray*}
&&\varphi(e)=k_1e+k_2f+k_3h,\\
&&\varphi(f)=k_4e+(q^2-q-1)k_1f+k_5h,\\
&&\varphi(h)=-2qk_5e-2qk_3f+(q-1)k_1h,
\end{eqnarray*}
and the coefficients $k_i\in\mathbb C$, $1\leq i\leq5$, satisfy
\begin{eqnarray}
\label{condition-1}
    &&(q-1)(q+1)(q-2)k_1=0,\\
\label{condition-2}
    &&(q+1)k_2=(q+1)k_4=0,\\
\label{condition-3}
    &&(q^2-1)k_3=(q^2-1)k_5=0.
\end{eqnarray}
\end{lem}
\begin{proof}
    Write
    \begin{eqnarray*}
        \varphi(e)=k_1e+k_2f+k_3h,\;\;
\varphi(f)=k_4e+k_5f+k_6h,\;\;
\varphi(h)=k_7e+k_8f+k_9h,
    \end{eqnarray*}
    where $k_i\in\mathbb C$, $1\leq i\leq9$. Substituting $(a,b)=(e,f),(e,h),(f,h)$ into Eq. (\ref{Lie-q-Der}) with $q$ replaced by $\frac1q$, the conclusion follows by a direct calculation.
\end{proof}

\begin{pro}\label{pro:invertible-q-derivations}
    An invertible $\frac1q$-derivation of $(\mathfrak{sl}_2(\mathbb C),[\cdot,\cdot])$ exists only when $q=2$ or $q=-1$. Moreover, we have the following conclusions.
    \begin{itemize}
        \item[(a)] For $q=2$, all invertible $\frac12$-derivations are given by
        $$\varphi(e)=ke,\;\; \varphi(f)=kf,\;\; \varphi(h)=kh,$$
        where $k\in\mathbb C^\times$.
        \item[(b)] For $q=-1$, all invertible $(-1)$-derivations are given by
        \begin{eqnarray}\label{case:q=-1}
            \varphi(e)=k_1e+k_2f+k_3h,\;\;
\varphi(f)=k_4e+k_1f+k_5h,\;\;
\varphi(h)=2k_5e+2k_3f-2k_1h,
        \end{eqnarray}
        where $k_i\in\mathbb C$, $1\leq i\leq5$, and $-k_1^3+k_1k_2k_4-2k_1k_3k_5+k_4k_3^2+k_2k_5^2\neq0$.
    \end{itemize}
\end{pro}
\begin{proof}
    We use the notation in Lemma \ref{lem:q-derivation-equations}. If $q\neq1$, then Eq. (\ref{condition-1}) gives
    $$(q+1)(q-2)k_1=0.$$
    If $q\neq-1$, then Eq. (\ref{condition-2}) gives $$k_2=k_4=0.$$
    If $q\neq\pm1$, then Eq. (\ref{condition-3}) implies $$k_3=k_5=0.$$ Hence, for $q\neq1$, an invertible $\frac1q$-derivation of $(\mathfrak{sl}_2(\mathbb C),[\cdot,\cdot])$ can exist only when $q=-1$ or $q=2$.

    For $q=1$, $\varphi$ is a derivation of $(\mathfrak{sl}_2(\mathbb C),[\cdot,\cdot])$. Since every derivation of $(\mathfrak{sl}_2(\mathbb C),[\cdot,\cdot])$ is inner, we have $\varphi=\mathrm{ad}(s)$ for some $s\in\mathfrak{sl}_2(\mathbb C)$. But $\mathrm{ad}(s)s=0$. Hence $\varphi$ cannot be invertible.

    For $q=2$, we have
    $$\varphi(e)=k_1e,\;\; \varphi(f)=k_1f,\;\; \varphi(h)=k_1h.$$
    So $\mathrm{det}(\varphi)=k_1^3$, and $\varphi$ is invertible if and only if $k_1\neq0$. This gives Item (a).

    Finally, for $q=-1$, we get Eq. (\ref{case:q=-1}). A direct computation shows that $\mathrm{det}(\varphi)=2(-k_1^3+k_1k_2k_4-2k_1k_3k_5+k_4k_3^2+k_2k_5^2)$. Thus $\varphi$ is invertible if and only if $-k_1^3+k_1k_2k_4-2k_1k_3k_5+k_4k_3^2+k_2k_5^2\neq0$. This gives Item (b).
\end{proof}

\begin{thm}\label{thm:classification-sl2}
    There exists a compatible $q$-pre-Lie algebra structure on $(\mathfrak{sl}_2(\mathbb C),[\cdot,\cdot])$ if and only if $q=2$ or $q=-1$. Moreover, we have the following conclusions.
    \begin{itemize}
        \item[(a)] For $q=2$, the compatible $2$-pre-Lie algebra structure is unique and is given by
        \begin{eqnarray*}
            a\circ b=\frac12[a,b],\;\; a,b\in\mathfrak{sl}_2(\mathbb C).
        \end{eqnarray*}
        \item[(b)] For $q=-1$, every compatible  $-1$-pre-Lie algebra structure is of the form
        \begin{eqnarray}\label{eq:minus-one-multiplication}
            a\circ_\varphi b=-\varphi^{-1}([a,\varphi(b)]),
\;\; a,b\in\mathfrak{sl}_2(\mathbb C),
        \end{eqnarray}
        where $\varphi$ is an invertible $(-1)$-derivation of $(\mathfrak{sl}_2(\mathbb C),[\cdot,\cdot])$ given by Eq. (\ref{case:q=-1}) with $k_i\in\mathbb C$, $1\leq i\leq5$, and $-k_1^3+k_1k_2k_4-2k_1k_3k_5+k_4k_3^2+k_2k_5^2\neq0$. Furthermore, two such $(-1)$-derivations give the same multiplication if and only if they are nonzero scalar multiples of each other.
    \end{itemize}
\end{thm}
\begin{proof}
    The conclusions follow directly from Propositions \ref{prop:sl2-existence-parametrization}, \ref{prop:sl2-same-structure} and \ref{pro:invertible-q-derivations}.
\end{proof}

\begin{ex}
Let $\varphi(e)=-e$, $\varphi(f)=-f$ and $\varphi(h)=2h$. Then the linear map $\varphi$ is an invertible $(-1)$-derivation. Then the corresponding compatible $(-1)$-pre-Lie algebra structure on $\mathfrak{sl}_2(\mathbb C)$ is given by the following nonzero products
\begin{eqnarray*}
h\circ e=-2e,\;\;e\circ h=-4e,\;\;h\circ f=2f,\;\;f\circ h=4f,\;\;e\circ f=\frac{1}{2}h,\;\;f\circ e=-\frac{1}{2}h.
\end{eqnarray*}
Such compatible root-graded anti-pre-Lie algebra structure on $\mathfrak{sl}_2(\mathbb C)$ was given in \cite{LB, BG2}.

Let $\varphi(e)=-e$, $\varphi(f)=e-f$ and $\varphi(h)=2h$. Then the linear map $\varphi$ is an invertible $(-1)$-derivation. Then the corresponding compatible $(-1)$-pre-Lie algebra structure on $\mathfrak{sl}_2(\mathbb C)$ is given by the following nonzero products
\begin{eqnarray*}
&&e\circ f=\frac{1}{2}h,\;\;e\circ h=-4e,\;\;f\circ e=-\frac{1}{2}h,\;\;f\circ f=\frac{1}{2}h,\;\;\\
&&f\circ h=4e+4f,\;\;h\circ e=-2e,\;\;h\circ f=4e+2f.
\end{eqnarray*}
Note that this compatible anti-pre-Lie algebra is not root-graded in the sense of \cite{BG2}.
\end{ex}
\begin{pro}
    Let $\varphi_1$ and $\varphi_2$ be two invertible $(-1)$-derivations of $(\mathfrak{sl}_2(\mathbb C),[\cdot,\cdot])$. Then the compatible $(-1)$-pre-Lie algebra structures $(\mathfrak{sl}_2(\mathbb C),\circ_{\varphi_1})$ and $(\mathfrak{sl}_2(\mathbb C),\circ_{\varphi_2})$ are isomorphic if and only if there exist $\lambda\in\mathbb C^\times$ and $\psi\in\mathrm{Aut}(\mathfrak{sl}_2(\mathbb C),[\cdot,\cdot])$ such that
    \begin{eqnarray}\label{eq:isomorphism-criterion}
\varphi_2=\lambda\psi\varphi_1\psi^{-1}.
\end{eqnarray}
\end{pro}
\begin{proof}
    Suppose that $(\mathfrak{sl}_2(\mathbb C),\circ_{\varphi_1})\cong(\mathfrak{sl}_2(\mathbb C),\circ_{\varphi_2})$. Then there exists a linear isomorphism $\psi:(\mathfrak{sl}_2(\mathbb C),\circ_{\varphi_1})\to(\mathfrak{sl}_2(\mathbb C),\circ_{\varphi_2})$ such that
    \begin{eqnarray}\label{psi-iso}
        \psi(a\circ_{\varphi_1}b)
=\psi(a)\circ_{\varphi_2}\psi(b),
\;\; a,b\in\mathfrak{sl}_2(\mathbb C).
    \end{eqnarray}
    So we have
    \begin{eqnarray*}
        \psi([a,b])=\psi(a\circ_{\varphi_1}b-b\circ_{\varphi_1}a)=\psi(a)\circ_{\varphi_2}\psi(b)-\psi(b)\circ_{\varphi_2}\psi(a)=[\psi(a),\psi(b)],\;\;a,b\in\mathfrak{sl}_2(\mathbb{C}),
    \end{eqnarray*}
    that is, $\psi\in\mathrm{Aut}(\mathfrak{sl}_2(\mathbb C),[\cdot,\cdot])$. Moreover, by Eqs. (\ref{eq:minus-one-multiplication}) and (\ref{psi-iso}), we obtain
    \begin{eqnarray*}
        0=\psi(a\circ_{\varphi_1}b)
-\psi(a)\circ_{\varphi_2}\psi(b)=-\psi\Big(\varphi_1^{-1}\big(\mathrm{ad}(a)(\varphi_1(b))\big) \Big)+\varphi_2^{-1}\Big(\mathrm{ad}(\psi(a))\big(\varphi_2(\psi(b))\big)  \Big),
    \end{eqnarray*}
    for any $a,b\in\mathfrak{sl}_2(\mathbb{C})$. Hence
    \begin{eqnarray*}
        \psi\varphi_1^{-1}\mathrm{ad}(a)\varphi_1=\varphi_2^{-1}\mathrm{ad}(\psi(a))\varphi_2\psi=\varphi_2^{-1}\psi\mathrm{ad}(a)\psi^{-1}\varphi_2\psi ,\;\;a\in\mathfrak{sl}_2(\mathbb{C}).
    \end{eqnarray*}
    Thus $\psi^{-1}\varphi_2\psi\varphi_1^{-1} $ commutes with $\mathrm{ad}(a)$ for any $a\in\mathfrak{sl}_2(\mathbb{C})$. By Schur's lemma, there exists $\lambda\in\mathbb C^\times$ such that $\psi^{-1}\varphi_2\psi\varphi_1^{-1}=\lambda\mathrm{id}_{\mathfrak{sl}_2(\mathbb{C})} $, which implies that Eq. (\ref{eq:isomorphism-criterion}) holds.

    Conversely, suppose that Eq. (\ref{eq:isomorphism-criterion}) holds for some $\lambda\in\mathbb C^\times$ and $\psi\in\mathrm{Aut}(\mathfrak{sl}_2(\mathbb C),[\cdot,\cdot])$. Then by Eq. (\ref{eq:minus-one-multiplication}), we get
    \begin{eqnarray*}
        \psi(a)\circ_{\varphi_2}\psi(b)&=&-\varphi_2^{-1}([\psi(a),\varphi_2(\psi(b))])=-\lambda^{-1}\psi\Big(\varphi_1^{-1}\big(\psi^{-1}([\psi(a),\lambda\psi(\varphi_1(b))])\big)\Big)\\
        &=&-\psi\big(\varphi_1^{-1}([a,\varphi_1(b)]) \big)=\psi(a\circ_{\varphi_1}b ),
    \end{eqnarray*}
    for any $a,b\in \mathfrak{sl}_2(\mathbb C)$. Therefore, $(\mathfrak{sl}_2(\mathbb C),\circ_{\varphi_1})\cong(\mathfrak{sl}_2(\mathbb C),\circ_{\varphi_2})$. This completes the proof.
\end{proof}

\noindent {\bf Acknowledgments.} This research is supported by Zhejiang
Provincial Natural Science Foundation of China (No. LZ25A010004) and
Natural Science Foundation of China (No. 12171129).

\smallskip

\noindent
{\bf Declaration of interests. } The authors have no conflicts of interest to disclose.

\smallskip

\noindent
{\bf Data availability. } No new data were created or analyzed in this study.

\end{document}